\documentclass[reqno,12pt,letterpaper]{amsart}
\usepackage{amsmath,amssymb,amsthm,graphicx,mathrsfs,url}
\usepackage[usenames,dvipsnames]{color}
\usepackage[colorlinks=true,linkcolor=Red,citecolor=Green]{hyperref}
\usepackage{floatrow}
\usepackage[usenames,dvipsnames]{xcolor}
\usepackage{tikz} 
\usepackage{slashed}
\usepackage{graphicx}
\usepackage[colorlinks=true]{hyperref}
\usepackage{color}
\usepackage{amsmath,amsfonts}
\usepackage{calrsfs}
\usepackage{amsmath,environ}
\usepackage{floatrow}
\DeclareMathAlphabet{\pazocal}{OMS}{zplm}{m}{n}
\usepackage{autonum}
\usepackage{bbm}

\usepackage{amssymb}
\usepackage{bbm}

\numberwithin{equation}{section}
\usepackage{amsmath, amsthm}
    \usepackage{dsfont}
    \usepackage{fancybox}
    \usepackage{amssymb}

\newcommand{\Di}{\slashed{\Dd}}

\newcommand{\tmu}{{\mu_\sharp}}

\newcommand{\III}{{\pazocal{I}}}
\newcommand{\QQ}{{\pazocal{Q}}}
\newcommand{\QQQ}{{\mathscr{Q}}}
\newcommand{\VVV}{{\mathscr{V}}}
\newcommand{\KKK}{{\mathscr{K}}}

\newcommand{\R}{\mathbb{R}}
\newcommand{\N}{\mathbb{N}}
\newcommand{\NN}{\mathcal{N}}
\newcommand{\p}{{\partial}}

\newcommand{\dd}[2]{\dfrac{\partial #1}{\partial #2}}
\newcommand{\matrice}[1]{\left[ \begin{matrix}
#1
\end{matrix} \right]}
\newcommand{\Det}{{\operatorname{Det}}}
\newcommand{\PP}{{\pazocal{P}}}
\newcommand{\PPP}{{\mathscr{P}}}

\newcommand{\loc}{{\text{loc}}}
\newcommand{\epsi}{\varepsilon}
\newcommand{\sgn}{{\operatorname{sgn}}}
\newcommand{\trace}{{\operatorname{Tr}}} 
\newcommand{\Tr}{{\operatorname{Tr}}}

\newcommand{\te}{\theta}
\newcommand{\lr}[1]{\langle #1 \rangle}
\newcommand{\blr}[1]{\left\langle #1 \right\rangle}

\newcommand{\GG}{\pazocal{G}}

\newcommand{\Bb}{\mathbb{B}}

\newcommand{\Z}{\mathbb{Z}}
\newcommand{\Dd}{\mathbb{D}}
\newcommand{\Ii}{\mathbb{I}}
\newcommand{\Jj}{\mathbb{J}}
\newcommand{\Ll}{\mathbb{L}}

\newcommand{\Pp}{\mathbb{P}}
\newcommand{\Ww}{\mathbb{W}}

\newcommand{\Id}{{\operatorname{Id}}}
\newcommand{\BB}{\mathbbm{b}}

\newcommand{\CCC}{\pazocal{C}}

\newcommand{\II}{\pazocal{I}}

\newcommand{\MM}{\pazocal{M}}

\newcommand{\UU}{{\pazocal{U}}}
\newcommand{\KK}{{\pazocal{K}}}

\newcommand{\HH}{\pazocal{H}}
\newcommand{\SSS}{{\pazocal{S}}}

\newcommand{\ove}[1]{{\overline{#1}}}
\newcommand{\systeme}[1]{\left\{ \begin{matrix} #1 \end{matrix} \right.}

\newcommand{\C}{\mathbb{C}}

\newcommand{\az}{\alpha}

\newcommand{\Vv}{\mathbb{V}}

\newcommand{\OO}{{\mathscr{O}}}

\newcommand{\dist}{{\operatorname{dist}}}

\newcommand{\Sf}{{\operatorname{Sf}}}

\newcommand{\ess}{{\operatorname{ess}}}

\newcommand{\Mm}{\mathbb{M}}

\newcommand{\Gg}{\mathbb{G}}
\newcommand{\Aa}{\mathbb{A}}

\newcommand{\Tt}{{\mathbb{T}}}
\newcommand{\Ee}{{\mathbb{E}}}

\newcommand{\RR}{\pazocal{R}}
\newcommand{\de}{ \ \mathrel{\stackrel{\makebox[0pt]{\mbox{\normalfont\tiny def}}}{=}} \ }

\title[The bulk-edge correspondence for continuous honeycomb lattices]{The bulk-edge correspondence for continuous honeycomb lattices}
\author{Alexis Drouot}
\NewEnviron{equations}{%
\begin{equation}\begin{gathered}
  \BODY
\end{gathered}\end{equation}
}
\newtheorem{thm}{Theorem}

\newtheorem{defi}{Definition}
\newtheorem{lem}{Lemma}[section]

\newtheorem{theorem}[thm]{Theorem}
\theoremstyle{definition}
\newtheorem{rmk}{Remark}[section]

\begin{document}
\vspace*{-1cm}
\maketitle
\thispagestyle{empty}

\newcommand{\tB}{\widetilde{B}}
\newcommand{\tlambda}{\widetilde{\lambda}}

\vspace*{-9mm}

\begin{abstract} We study bulk/edge aspects of continuous honeycomb lattices in a magnetic field. We compute the bulk index of Bloch eigenbundles: it equals $2$ or $-2$, with sign depending on nearby Dirac points and on the magnetic field. We then prove the existence of two topologically protected unidirectional waves propagating along line defects. This shows the  bulk/edge correspondence for our class of Hamiltonians.
\end{abstract}

\section{Introduction}

This note focuses on bulk/edge aspects of \textit{continuous, asymptotically periodic} Hamiltonians $\Pp_\delta$. These operators model electronic transport between honeycomb lattices, when a magnetic field breaks time-reversal symmetry. Related models have suggested an analogy between photonic structures and topological insulators \cite{HR,RH}. 

In an asymptotic regime, \cite{FLTW3,LWZ} mathematically constructed edge states bifurcating from Dirac points energies. Here, we relate their existence to a non-zero bulk invariant. This demonstrates their persistence outside the perturbative regime.

\subsection{Bulk index}\label{sec:1.1} The bulk operators associated to $\Pp_\delta$ are
\begin{equation}\label{eq:2c}
\Pp_{\delta,+} \de -\Delta_{\R^2} + \Vv + \delta \cdot \Ww, \ \ \ \  \ \Pp_{\delta,-} \de -\Delta_{\R^2} - \Vv - \delta \cdot \Ww, \ \ \  \text{ where: }
\end{equation} 
\begin{itemize}
\item $\Vv \in C^\infty\big(\R^2,\R\big)$ is even, periodic with respect to the equilateral lattice $\Lambda$ and invariant under the $2\pi/3$-rotation $R$ -- see \S\ref{sec:2.1}.
\item $\Ww = \frac{1}{i} \big( \Aa \cdot \nabla + \nabla \cdot \Aa \big)$  and $\Aa \in C^\infty\big(\R^2,\R^2\big)$ is odd and periodic w.r.t. $\Lambda$.
\end{itemize} 

 Our first result computes the bulk index of  $\Pp_{\delta,\pm}$. Let $\lambda_{\delta,1}(\xi) \leq \dots \leq \lambda_{\delta,j}(\xi) \leq \dots$ be the eigenvalues of $\Pp_{\delta,+}$ on Floquet spaces $L^2_\xi$ -- see \S\ref{sec:2.1}. Generically, $\Pp_{0,+}$ admits Dirac points $(\xi_\star,E_\star)$ and $(-\xi_\star,E_\star)$ -- see  \cite{FW}. These come with pairs $(\phi_1,\phi_2) \in L^2_{\xi_\star}$ with
\begin{equation}
\ker_{L^2_{\xi_\star}}(\Pp_{0,+}-E_\star) = \C\phi_1 \oplus \C\phi_2,  \ \ \ \ 
\phi_1(Rx) = e^{2i\pi/3} \phi_1(x), \ \ \ \ \phi_2(Rx) = e^{-2i\pi/3} \phi_2(x).
\end{equation}
We write $E_\star = \lambda_{0,n}(\xi_\star)=\lambda_{0,n+1}(\xi_\star)$ and we assume:
\begin{equations}\label{eq:2a}
\lambda_{0,n}(\xi) = \lambda_{0, n+1}(\xi) \ \ \Leftrightarrow \ \ \xi \in \big\{\xi_\star,-\xi_\star\big\} \ \ \text{modulo the dual lattice $2\pi \Lambda^*$};
\end{equations}
\begin{equations}\label{eq:1a}
\hspace{-2mm} \te_\star \de \lr{\phi_1, \Ww \phi_1}_{L^2_{\xi_\star}} \neq 0; \ \ \ \text{and} \  \ \  \inf \big\{\delta > 0 : \exists \xi \in \R^2, \  \lambda_{\delta,n}(\xi) = \lambda_{\delta,n+1}(\xi) \big\} > 0.
\end{equations}
Let $\delta_\sharp$ be the infimum in \eqref{eq:1a}.
This assumption allows to construct a smooth bundle $\Ee_{\delta,\pm}$ over the torus $\R^2/(2\pi\Lambda^*)$ when $\delta \in (0,\delta_\sharp)$: the fiber at $\xi$ is the $L^2_\xi$-eigenspace of $\Pp_{\delta,+}$ corresponding to $\lambda_{\delta,1}(\xi), \dots,$ $\lambda_{\delta,n}(\xi)$. Following the physics literature, the bulk index of $\Pp_{\delta,\pm}$ is the first Chern class $c_1(\Ee_{\delta,\pm})$ -- see \S\ref{sec:2.2a}. We similarly define $c_1(\Ee_{\delta,-})$.

\begin{theorem}\label{thm:1} Assume that \eqref{eq:2a} and \eqref{eq:1a} hold. Then for every $\delta \in (0,\delta_\sharp)$,
\begin{equation}
c_1(\Ee_{\delta,+}) = - \sgn(\te_\star) \  \text{ and  } \ c_1(\Ee_{\delta,-}) = \sgn(\te_\star).
\end{equation}
\end{theorem}

\subsection{Edge index and the bulk-edge correspondence}\label{sec:1.2} Let $v \in \Lambda$ -- representing the direction of an edge. The operator  $\Pp_\delta$ considered in \cite{Dr0} is
\begin{equation}
\Pp_\delta \de -\Delta_{\R^2} + \Vv + \delta \cdot \kappa_\delta \cdot \Ww.
\end{equation}
Above, $\kappa_\delta$ is a domain wall across $\R v$: there exists  $\kappa \in C^\infty(\R,\R)$ equal to $\pm 1$ near $\pm \infty$ such that $\kappa_\delta(x) = \kappa(\delta \lr{k',x})$, where $k' \in \Lambda^*$ is dual to $v$ -- see \S\ref{sec:3.1}. See \cite{HR,LWZ} for related models. The operators $\Pp_{\delta,\pm}$ in \eqref{eq:2c} are the limits of $\Pp_\delta$ as $\lr{k',x} \rightarrow \pm\infty$.

The operator $\Pp_\delta$ is not a periodic operator with respect to $\Lambda$. It is however periodic with respect to $\Z v$. For $\zeta \in \R$, let $\Pp_\delta[\zeta]$
 be the operator equal to $\Pp_\delta$, but acting on
\begin{equation}
L^2[\zeta] \de \left\{ u \in L^2_\loc\big(\R^2,\C\big), \ u(x+v) = e^{i\zeta} u(x),  \ \ \int_{\R^2/\Z v} |u(x)|^2 dx   < \infty\right\}.
\end{equation}
Fix $\delta_\flat \in (0,\delta_\sharp)$ and assume that there exists $E_\flat \in C^\infty\big(\R/(2\pi\Z), \R\big)$ with
\begin{equation}\label{eq:2b}
\forall \zeta, \tau, \tau' \in [0,2\pi], \ \ \ \ \lambda_{\delta_\flat,n}(\zeta k + \tau k') < E_\flat(\zeta) <   \lambda_{\delta_\flat,n+1}(\zeta k + \tau' k').
\end{equation}
Then for every $\zeta \in \R$, $E_\flat(\zeta)$ is not in the essential spectrum of $\Pp_{\delta_\flat}[\zeta]$. This allows to define the edge index $\NN$ of $\Pp_{\delta_\flat}$ as the spectral flow of $\Pp_{\delta_\flat}-E_\flat$. It is the signed number of eigenvalues of 
$\Pp_{\delta_\flat}[\zeta]$ that cross the gap containing $E_\flat(\zeta)$ downwards as $\zeta$ sweeps $[0,2\pi]$. See \cite{Wa} for an introduction to spectral flow.

\begin{thm}\label{thm:2} Assume that \eqref{eq:2a}, \eqref{eq:1a} and \eqref{eq:2b} hold, and that $\lr{\xi_\star,v} \notin \pi \Z$. Then
\begin{equation}\label{eq:2d}
\NN = c_1(\Ee_{\delta_\flat,+}) - c_1(\Ee_{\delta_\flat,-}) = -2 \cdot \sgn(\te_\star).
\end{equation}
\end{thm} 

Because the spectral flow is a topological invariant, Theorem \ref{thm:2} is stable under gap-preserving perturbations of $\Pp_\delta[\zeta]$. The condition $\lr{\xi_\star,v} \notin \pi \Z$ excludes armchair-type edges; we will deal with such edges in an upcoming work. 

When $\delta \ll 1$, edge states of operators similar to $\Pp_\delta$ were constructed in \cite{FLTW3,LWZ,Dr0} under the no-fold condition. This condition requires that the dispersion surfaces $\xi \mapsto \lambda_{0,n}(\xi)$ and $\xi \mapsto \lambda_{0,n+1}(\xi)$ do not fold over $E_\star$ except at $\{\xi_\star,-\xi_\star\} + 2\pi \Lambda^*$ -- see \cite[\S1.3]{FLTW3}. Theorem \ref{thm:2} implies that if $\Pp_{\delta_\flat}[\zeta]$ has a continuously open gap for every $\zeta \in [0,2\pi]$, two edge state must exist, even when the no-fold condition \textit{at $\delta=0$} fails. These edge states shall arise from resonant states bifurcating into the edge of the continuous spectrum -- see \cite[\S1.4]{FLTW3} and the conjecture there.

Theorem \ref{thm:2} is an index-like theorem that relates a topological index (the Chern number) to an analytic index (the spectral flow). It is the bulk-edge correspondence, an ubiquitous result in mathematical physics \cite{Hat,KRS,EG,GP,ASV,PS,BKR,Ku,Br}. Theorem \ref{thm:2} advances the current understanding~via:
\begin{itemize}
\item An analysis on a \textit{continuous, asymptotically periodic} model; see also \cite{KS,KS2,Taa} for the quantum Hall effect; \cite{Ba2} for Dirac operators; \cite{BR} for a K-theoretic approach; and \cite{Dr2} for dislocation systems. 
\item The explicit formula \eqref{eq:2d} for the bulk/edge indexes, which demonstrates the significance of Dirac points (or more generally degeneracies in the Bloch bands) in the production of topologically protected edge states. 
\end{itemize}

It would be interesting to investigate the validity of Theorem \ref{thm:1} and Theorem \ref{thm:2} for $\delta$ outside $(0,\delta_\sharp)$ or when disorder is added. This regime corresponds to the absence of mobility gap. Possible directions are suggested by \cite{EGS,Taa,GS,GT,ST}. 

Theorem \ref{thm:2} demonstrates the existence of topologically stable time-harmonic waves propagating along line defects in graphene. A recent analysis on Dirac operators \cite{Ba} suggests that these waves should be insensitive to back-scattering by local obstacles. We plan to mathematically analyze this phenomena.

\subsection{Sketches of proofs} The proof of Theorem \ref{thm:1} relies on three main steps:
\begin{itemize}
\item As a topological invariant, the Chern number does not depend on $\delta \in (0,\delta_\sharp)$: in this range, the bundles $\Ee_{\delta,\pm}$ are diffeomorphic to one another. Hence it suffices to compute $c_1(\Ee_{\delta,\pm})$ for small $\delta$ only. We then write the Chern number as the integral of the trace of the Berry curvature $\Bb_\delta(\xi)$. This formula involves the projector $\Pi_\delta(\xi)$ to the $n$-th lowest-energy eigenspaces of $\Pp_{\delta,+}(\xi)$.
\item When $\xi$ is away from Dirac momenta $\{\xi_\star,-\xi_\star\} + 2\pi \Lambda^*$, $\Pi_\delta(\xi)$ (and its derivatives) converges uniformly to $\Pi_0(\xi)$.  Hence $\Bb_\delta(\xi)$ converges uniformly to $\Bb_0(\xi)$. Because of symmetries, $\Bb_0(\xi)=0$: momenta away from $\{\xi_\star,-\xi_\star\} + 2\pi \Lambda_*$ do not contribute to the Chern number.
\item For $\xi$ near $\xi_\star$, we show that after rescaling, $\Pp_{\delta,+}(\xi)$ converges in the resolvent sense to the two-band model $\Mm_\delta(\xi)$ studied in \cite{HR}. This convergence transfers to $\Pi_\delta(\xi)$ and its derivatives, hence to $\Bb_\delta(\xi)$. The last part of the proof computes the Berry curvature and Chern number associated to the low-lying eigenbundle of $\Mm_\delta(\xi)$, eventually leading to $c_1(\Ee_{\delta,+}) = - \sgn(\te_\star)$.
\end{itemize}

In \cite[Corollary 4]{Dr0}, we showed that under a condition weaker than \eqref{eq:2b}, the edge index of $\Pp_\delta$ equals $-2 \cdot \sgn(\te_\star)$. That proof relied on a resolvent estimate for $\Pp_\delta[\zeta]$. Theorem \ref{thm:2} holds more generally. It differs from \cite[Corollary 4]{Dr0} because it applies to cases where the no-fold condition  fails.  This failure is an obstacle to construct edge states. The existence of long-lived states was instead conjectured \cite[\S1.4]{FLTW3}. 

In the setting of Theorem \ref{thm:2}, the operator $\Pp_{\delta_\flat}[\zeta]$ has an essential $L^2[\zeta]$-gap, but $\Pp_{\delta}[\zeta]$ may not have an essential spectral gap for small $\delta$. Therefore, the resolvent estimate \cite[Theorem 2]{Dr0} does not hold. 
In order to nonetheless prove Theorem \ref{thm:2}, we construct a modified operator $\PP_\delta$, with three essential properties:
\begin{itemize}
\item It has the same spectral flow as $\Pp_\delta-E_\flat$ when $\delta = \delta_\flat$;
\item It looks like $\Pp_\delta$ for momenta / energy near $(\xi_\star,E_\star)$ and $\delta$ near $0$.
\item It retains an essential $L^2[\zeta]$-gap as $\delta \in (0,\delta_\flat]$ decreases to $0$.
\end{itemize}

We can then apply the techniques of \cite{Dr0} to compute the spectral flow of $\PP_\delta-E_\star$. This relies on a resolvent expansion of $\PP_\delta$. There are two main steps:
\begin{itemize}
\item We use the limiting two-band model $\Mm_\delta(\xi)$ to approach $\PP_{\delta,+}(\xi)$, and we integrate these estimates to expand the bulk resolvents $(\PP_{\delta,+}[\zeta]-z)^{-1}$.
\item We construct a parametrix based on the bulk operators $\PP_{\delta,+}[\zeta]$. 
\end{itemize}
A family of Dirac operators -- which quantizes the limiting two-band model near infinity -- controls the effective dynamics near each of the two Dirac points. Each family has spectral flow equal to $-\sgn(\te_\star)$, which implies that $\PP_\delta$ has spectral flow $-2 \cdot \sgn(\te_\star)$.

These Dirac operators arised in previous work \cite{FLTW1,FLTW3,LWZ} where they were used to construct \textit{some} edge states as adiabatic modulations of the Dirac point Bloch modes. These constructions rely on a sophisticated Lyapounov--Schmidt reduction combined with multiscale analysis. Working at the level of the resolvent has the advantage of producing \textit{all} edge states. This knowledge is necessary to compute the edge index -- see \cite{DFW,Dr2} for bulk/edge analysis of dislocated models.

\subsection*{Notations} We will use the following notations:
\begin{itemize}\item $\Dd(z,r) \subset \C$ denotes the disk centered at $z \in \C$, of radius $r$.  
\item If $\HH$ is a Hilbert space and $A : \HH \rightarrow \HH$ is bounded, the norm of $A$ is
\begin{equation}
\| A \|_{\HH} \de \sup_{|\psi|_\HH=1} |A\psi|_{\HH}.
\end{equation}
\item If $A_\epsi : \HH \rightarrow \HH$ is a bounded operator  and $f : (0,\epsi_0] \rightarrow \R$, we write $A_\epsi = \OO_{\HH}\big(f(\epsi)\big)$ when there exists $C > 0$ such that $\|A_\epsi\|_{\HH} \leq C f(\epsi)$ for $\epsi \in (0,\epsi_0]$. 
\item If $\zeta \in \R/(2\pi \Z) \mapsto H(\zeta)$ is a continuous family of selfadjoint operators such that $0 \notin \Sigma_\ess\big(H(\zeta)\big)$, $\Sf(H)$ denotes the spectral flow of $H$ through zero as $\zeta$ spans $[0,2\pi]$ -- see \cite{Wa} for a comprehensive introduction.
\end{itemize}

\subsection*{Acknowledgments} I am grateful to M. I. Weinstein and J. Shapiro for valuable discussions. I thankfully acknowledge support from the Simons Foundation through M. I. Weinstein's Math+X investigator award $\#$376319 and from NSF DMS-1800086.

\section{Proof of Theorem \ref{thm:1}}

\subsection{Dirac points amd their bifurcations}\label{sec:2.1} Here we review honeycomb Schr\"odinger operators, Dirac points and gap openings via conjugation symmetry breaking. Let $\Lambda = \Z v_1 \oplus \Z v_2$ be the equilateral $\Z^2$-lattice:
\begin{equation}\label{eq:4b}
v_1 = a\matrice{\sqrt{3} \\ 1}, \ \ \ v_2= a\matrice{\sqrt{3} \\ -1},
\end{equation}
where  $a > 0$ is a constant such that $\Det[v_1,v_2] = 1$. The dual basis $k_1, k_2$ consists of two vectors in $(\R^2)^*$ which satisfy $\blr{k_i,v_j} = \delta_{ij}$. The dual lattice is $\Lambda^* = \Z k_1 \oplus \Z k_2$. The corresponding fundamental cell and dual fundamental cell are 
\begin{equation}\label{eq:4c}
\Ll \de \big\{ sv_1 + s' v_2 : s, s' \in [0,1)  \big\}, \ \ \ \ \Ll^* \de \big\{ \tau k_1 + \tau' k_2 : \tau, \tau' \in [0,2\pi)  \big\}.
\end{equation}

Honeycomb potential are smooth functions $\R^2 \rightarrow \R$ that are even, $\Lambda$-periodic and invariant under $2\pi/3$-rotations -- see \cite[Definition 1]{Dr0}. Let $\Vv$ be a honeycomb potential and $\Pp_0 = -\Delta_{\R^2}+\Vv$. Since $\Pp_0$ is periodic w.r.t. $\Lambda$, it acts on 
\begin{equation}
L^2_\xi \de \Big\{ u \in L^2_\loc(\R^2,\C) : u(x+w) = e^{i\lr{\xi,w}} u(x)\Big\},
\end{equation}
for all $\xi \in \R^2$.
We denote by $\Pp_0(\xi)$ the operator $\Pp_0$ acting on $L^2_\xi$; it has discrete spectrum $\lambda_{0,1}(\xi) \leq \dots \leq \lambda_{0,j}(\xi) \leq \dots$.

\begin{defi}\label{def:1} A pair $(\xi_\star,E_\star) \in \R^2 \times \R$ is a Dirac point of $\Pp_0$ if:
\begin{itemize}
\item[(i)] $E_\star$ is a $L^2_{\xi_\star}$-eigenvalue of $\Pp_0(\xi_\star)$
of multiplicity $2$;
\item[(ii)] There exists an orthonormal basis $\{\phi_1,\phi_2\}$ of $\ker_{L^2_{\xi_\star}}\hspace*{-1.5mm}\big(P_0(\xi_\star)-E_\star\big)$ such that 
\begin{equation}\label{eq:5a}
\phi_1(Rx) = e^{2i\pi/3} \phi_1(x), \ \ \ \ \phi_2(x) = \ove{\phi_1(-x)}, \ \ \ \ \phi_2(Rx) = e^{-2i\pi/3} \phi_2(x).
\end{equation}
\item[(iii)] There exist $n \geq 1$ and $\nu_F > 0$ such that  for $\xi$ close to $\xi_\star$,
\begin{equations}
\lambda_{0,n}(\xi) = E_\star - \nu_F \cdot |\xi-\xi_\star| + O(\xi-\xi_\star)^2,
 \\
\lambda_{0,n+1}(\xi) = E_\star + \nu_F \cdot |\xi-\xi_\star| + O(\xi-\xi_\star)^2.
\end{equations}
\end{itemize}
\end{defi}

In a seminal paper \cite{FW}, Fefferman and Weinstein  showed that for a generic choice of $\Vv$, $\Pp_0$ admit Dirac points $(\xi_\star,E_\star)$. Different perspectives -- on the proof and on the context -- have appeared since \cite{Le,FLTW1,LWZ,BC,ABL,AHY}. Because of \eqref{eq:5a}, $\xi_\star \in \big\{\xi_\star^A,\xi_\star^B\big\} + 2\pi \Lambda^*$ with
\begin{equation}\label{eq:3y}
\xi_\star^A \de \dfrac{2\pi}{3} (2k_1+k_2), \ \ \ \ \xi_\star^B \de \dfrac{2\pi}{3} (k_1+2k_2) = -\xi_\star^A \mod 2\pi \Lambda^*.
\end{equation}
Since $\Pp_0$ is invariant under spatial inversion, $\big(\xi_\star^A,E_\star\big)$ is a Dirac point of $\Pp_0$ if and only if $\big(\xi_\star^B,E_\star\big)$ is another Dirac point of $\Pp_0$. In the rest of the paper, we assume that $(\xi_\star,E_\star)$ is a Dirac point of $\Pp_0$, associated to the $n$-th band, and such that \eqref{eq:2a} holds:
\begin{equation}\label{eq:4c}
\lambda_{0,n}(\xi) = \lambda_{0, n+1}(\xi) \ \ \Rightarrow \ \ \xi \in \big\{\xi_\star^A,\xi_\star^B\big\} + 2\pi \Lambda^*.
\end{equation}
We take $\{\phi_1,\phi_2\} \subset L^2_{\xi_\star}$ satisfying \eqref{eq:5a}.

Introduce the operators
\begin{equations}\label{eq:5d}
\Pp_{\delta,\pm} \de \Pp_0 \pm \delta \Ww = -\Delta_{\R^2} + \Vv \pm \delta \Ww, \ \ \ \ \Ww = \Aa \cdot D_x + D_x \cdot \Aa, \ \ \ \ D_x \de \dfrac{1}{i} \matrice{\p_{x_1} \\ \p_{x_2}},
\end{equations}
where $\Aa \in C^\infty\big(\R^2,\R^2\big)$ is periodic w.r.t. $\Lambda$ and $\Aa(x) = \Aa(-x)$. They are conjugation-breaking perturbations of $\Pp_0$; they represent graphene-like structures affected by a magnetic field. The work \cite{LWZ} considers other conjugation-breaking operators, and constructs edge states in a perturbative adiabatic regime. We define
\begin{equation}
\delta_\sharp \de \inf \big\{\delta > 0 : \exists \xi \in \R^2, \  \lambda_{\delta,n}(\xi) = \lambda_{\delta,n+1}(\xi) \big\}.
\end{equation}
For $\delta \in (0,\delta_\sharp)$, the $n$-th $L^2_\xi$-gap of $\Pp_{\delta,\pm}(\xi)$ is open:
\begin{equation}\label{eq:4a}
\inf_{\xi \in \R^2} \big(\lambda_{\delta,n+1}(\xi) - \lambda_{\delta,n}(\xi)\big) > 0.
\end{equation}
If \eqref{eq:2a} holds and $\te_\star = \lr{\phi_1,\Ww\phi_1}_{L^2_{\xi_\star}} \neq 0$ then $\delta_\sharp > 0$ -- see Lemma \ref{lem:1b}. This means that breaking conjugation invariance opens the $n$-th gap of $\Pp_\delta(\xi)$. 
In the rest of the paper, we work with $\delta \in (0,\delta_\sharp)$; in particular, \eqref{eq:4a} always holds.

\subsection{Bulk index}\label{sec:2.2a} We review the definition of bulk index. 
For $\delta \in (0,\delta_\sharp)$, the gap condition \eqref{eq:4a} holds. We can then define a rank-$n$ vector bundle $\Ee_{\delta,+}$ over the two-torus $\Tt^2 = \R^2/(2\pi \Lambda^*)$: the fiber at a point $\xi \in \Tt^2$ is the vector space
\begin{equation}
\Ee_{\delta, +}(\xi) \de \bigoplus_{j=1}^{n} \ker_{L^2_{\xi_\star}}\big( \Pp_{\delta,+}(\xi) - \lambda_{\delta,j}(\xi) \big) \subset L^2_{\xi}.
\end{equation}
When provided with its natural structure, this bundle is smooth because of the gap condition \eqref{eq:4a} and \cite[\S VII.1.3, Theorem 1.7]{Ka}. In order to define the bulk index, we first look at $\Ee_{\delta,+}$ as a bundle over $\R^2$ instead of $\Tt^2$. Since $\R^2$ is contractible, this bundle is trivial -- see \cite[pp. 15]{Mo}. Therefore it admits a smooth orthonormal frame
\begin{equation}\label{eq:3h}
\xi \in \R^2 \ \mapsto \ \big( \psi_{\delta,1}(\xi), \dots, \psi_{\delta,n}(\xi) \big) \in L^2_{\xi} \times \dots \times L^2_{\xi}.
\end{equation}
For every $\delta \in (0,\delta_\sharp)$, the orthogonal projector $\Pi_\delta(\xi) : L^2_\xi \rightarrow L^2_\xi$ onto $\Ee_{\delta,+}$ varies smoothly with $\xi$ -- this means that $e^{-i\lr{\xi, x}} \Pi_\delta(\xi) e^{i\lr{\xi, x}}$
forms a smooth family of operators on $L^2_0$. The operator $\Pi_\delta(\xi)$ relates to the orthonormal frame \eqref{eq:3h} via
\begin{equation}
\xi \in \R^2 \ \ \Rightarrow \ \ \Pi_\delta(\xi) = \sum_{j=1}^n \psi_{\delta,j}(\xi) \otimes \psi_{\delta,j}(\xi).
\end{equation}

Let $\Gamma\big(\R^2,\Ee_{\delta,+}\big)$ be the space of smooth sections of the bundle $\Ee_{\delta,+}$ over $\R^2$; fix $\sigma \in \Gamma\big(\R^2,\Ee_{\delta,+}\big)$. We write $\sigma = \sum_{j=1}^n \sigma_{\delta,j} \cdot \psi_{\delta,j}$; the coordinates $\sigma_{\delta,j}$ are smooth functions $\R^2 \rightarrow \C$.  For $\xi \in \R^2$, we set
\begin{equations}
\nabla \sigma(\xi) \de  \sum_{j=1}^n  d\sigma_{\delta,j}(\xi) \cdot \psi_{\delta,j}(\xi)   + \sigma_{\delta,j}(\xi) \cdot \Pi_\delta(\xi)  \big( \p \psi_{\delta,j}(\xi)  \big), \ \ \ \ \  \ \text{where}
\\
d\sigma_{\delta,j}(\xi) \de \sum_{m=1}^2 \dd{\sigma_{\delta,j}(\xi)}{\xi_m} \cdot d\xi_m, \ \ \ \ \p \psi_{\delta,j}(\xi) \de \sum_{m=1}^2 e^{i\lr{\xi,x}} \cdot \dd{\big(e^{-i\lr{\xi,x}}\psi_{\delta,j}(\xi)\big)}{\xi_m} \cdot d\xi_m.
\end{equations}
We observe that $\nabla \sigma$ is an element of $\Gamma\big(\R^2,\Ee_{\delta,+} \otimes \Tt^* \R^2\big)$.

If $\sigma \in \Gamma\big(\Tt^2,\Ee_{\delta,+}\big)$, we can see $\sigma$ as an element of $\Gamma\big(\R^2,\Ee_{\delta,+}\big)$. Then $\nabla \sigma$ happens to be  an element of $\Gamma\big(\Tt^2,\Ee_{\delta,+} \otimes T^*\Tt^2\big)$.\footnote{For a proof when $n=1$, we refer to \cite[\S4]{Dr2}; the same argument applies to $n > 1$. It relies on the fact that the frame $(\psi_{\delta,1}, \dots, \psi_{\delta,n}
)$ defines coordinates on $\Ee_{\delta,+}$ with unitary transition functions.} Since $\nabla$ satisfies Leibnitz's rule, $\nabla$ is a connection on the bundle $\Ee_{\delta,+} \rightarrow \Tt^2$. It is called the Berry connection; its curvature is the Berry curvature \cite{Si,Be}. 

The trace of the Berry curvature, $\Bb_\delta(\xi) d\xi_1 \wedge d\xi_2$ has an expression in terms of $\Pi_\delta(\xi)$ that is manifestly gauge-invariant, i.e. independent of the choice of frame in \eqref{eq:3h}: 
\begin{equations}\label{eq:3i}
\Bb_\delta(\xi) \de \Tr\Big( \Pi_\delta(\xi) \big[ \nabla_{\xi_1} \Pi_\delta(\xi), \nabla_{\xi_2} \Pi_\delta(\xi) \big] \Big) \ \ \ \  \text{  where} 
 \\
\nabla_{\xi_m} \Pi_\delta(\xi) \de e^{i\lr{\xi,x}} \cdot \dd{\left(e^{-i\lr{\xi,x}} \Pi_\delta(\xi) e^{i\lr{\xi,x}}\right) }{\xi_m} \cdot e^{-i\lr{\xi,x}} \ : \  L^2_\xi \rightarrow L^2_\xi.
\end{equations}
We mention that $\Bb_\delta(\xi)$ is purely imaginary. Indeed, using that $\Pi_\delta(\xi)$ is selfadjoint and that the trace is cyclic, $\ove{\Bb_\delta(\xi)}$ equals
\begin{equation}
\trace\Big(  \big[ \nabla_{\xi_2} \Pi_\delta(\xi)^*, \nabla_{\xi_1} \Pi_\delta(\xi)^* \big] \Pi_\delta(\xi)^* \Big) = \trace\Big( \Pi_\delta(\xi) \big[ \nabla_{\xi_2} \Pi_\delta(\xi), \nabla_{\xi_1} \Pi_\delta(\xi)\big] \Big) = -\Bb_\delta(\xi).
\end{equation}

The first Chern class of $\Ee_{\delta,+}$, or bulk index, is the integral of $\Bb_\delta(\xi)$ over $\Tt^2$:
\begin{equations}\label{eq:2n}
 \ c_1(\Ee_{\delta,+}) = \dfrac{i}{2\pi} \int_{\Tt^2} \Bb_\delta(\xi) \cdot d\xi.
\end{equations}
See for instance \cite[(15)]{Ba2}. This is a topological integer -- see e.g. \cite[pp. 49]{BGV}. In particular, it doe
s not depend on $\delta \in (0,\delta_\sharp)$. To prove Theorem \ref{thm:1} we will compute $c_1(\Ee_{\delta,+})$ in the limit $\delta \rightarrow 0$: because of topological invariance,
\begin{equation}
\delta \in (0,\delta_\sharp) \ \ \Rightarrow \ \ 
c_1(\Ee_{\delta,+}) = \lim_{\delta \rightarrow 0^+} c_1(\Ee_{\delta,+}).
\end{equation}

\subsection{Berry curvature of the unperturbed operator}\label{sec:2.2} Let $\Ee_0 \rightarrow \Tt^2$ be the bundle with fibers
\begin{equation}
\Ee_0(\xi) = \bigoplus_{j=1}^{n} \ker_{L^2_\xi}\big( \Pp_0(\xi) - \lambda_{0,j}(\xi) \big) \subset L^2_\xi, \ \ \ \ \Pp_0(\xi) = -\Delta_{\R^2} + \Vv : L^2_\xi \rightarrow L^2_\xi.
\end{equation}
If \eqref{eq:4a} holds, the restriction of $\Ee_0$ to $\Tt^2 \setminus \big(\{\xi_\star^A,\xi_\star^B\} + 2\pi\Lambda^*\big)$ is a smooth vector bundle of rank $n$ (when provided with its canonical structure). The trace of the Berry curvature $\Bb_0$ of this bundle is defined via \eqref{eq:3i}. We show here that $\Bb_0$ vanishes uniformly because $\Pp_0$ is invariant under both $\II$ (spatial inversion) and $\CCC$ (complex conjugation).

Since $\CCC \Pp_0(\xi) \CCC^{-1} = \Pp_0(-\xi)$, we deduce that 
$\CCC \Pi_0(\xi) \CCC^{-1} = \Pi_0(-\xi)$. It follows that 
\begin{equations}\label{eq:2e}
\ove{\Bb_0(\xi)} = \trace\Big( \CCC \cdot \Pi_0(\xi) \big[ \nabla_{\xi_1} \Pi_0(\xi), \nabla_{\xi_2} \Pi_0(\xi) \big]\cdot \CCC^{-1} \Big)
\\
 = \trace\Big( \Pi_0(-\xi) \big[ \nabla_{\xi_1} \Pi_0(-\xi), \nabla_{\xi_2} \Pi_0(-\xi) \big] \Big) = \Bb_0(-\xi).
\end{equations}
Since $\Bb_0(\xi)$ is a purely imaginary number, we deduce that $\Bb_0(-\xi) = -\Bb_0(\xi)$. Since $\II \Pp_0(\xi) \II^{-1} = \Pp_0(-\xi)$, $\III \Pi_0(\xi) \III^{-1} = \Pi_0(-\xi)$ and
\begin{equations}
\Bb_0(\xi) = \trace\Big( \III \cdot \Pi_0(\xi) \big[ \nabla_{\xi_1} \Pi_0(\xi), \nabla_{\xi_2} \Pi_0(\xi) \big] \cdot \III^{-1} \Big)
\\
 = \trace\Big( \Pi_0(-\xi) \big[ \nabla_{\xi_1} \Pi_0(-\xi), \nabla_{\xi_2} \Pi_0(-\xi) \big] \Big) = \Bb_0(-\xi).
\end{equations}
This shows that $\Bb_0(\xi) = \Bb_0(-\xi)$. We conclude that $\Bb_0(\xi) = 0$. 

\begin{rmk}\label{rmk:1} Fefferman--Lee-Thorp--Weinstein \cite{FLTW3} studied an operator $\PPP_\delta$ that shares many of the characteristics of $\Pp_\delta$, but that is invariant under spatial inversion instead of complex conjugation. They produced two edge states for $\PPP_\delta[\zeta]$ as adiabatic combinations of the Dirac point Bloch modes with eigenvectors of an emerging Dirac operator. The associated time-harmonic waves propagate in opposite directions. In \cite{Dr0} we proved that all edge states take this form, and showed that the corresponding spectral flow vanishes. This agrees with the bulk-edge correspondence. Indeed, the bulk operators are invariant under $\CCC$: the trace of the Berry curvature is odd -- see \eqref{eq:2e}. Since the Chern number is the integral of the Berry curvature, it vanishes.
\end{rmk}

\subsection{Away from Dirac momenta} In this section we study $\Bb_\delta(\xi)$ when $\delta$ is small and $\xi$ is away from $\big\{\xi_\star^A,\xi_\star^B\big\} + 2\pi \Lambda^*$. Define
\begin{equation}
\rho (\xi) \de \dist\left( \xi, \big\{\xi_\star^A,\xi_\star^B\big\} + 2\pi \Lambda^*  \right).
\end{equation}

\begin{lem}\label{lem:1a} Under \eqref{eq:2a}, for every $\epsi > 0$, there exists $C > 0$ such that
\begin{equations}
\delta \in [0,\delta_\sharp) \ \ \Rightarrow \ \ 
\sup\big\{ |\Bb_\delta(\xi)| : \ \rho(\xi) \geq \epsi \big\} \leq C\delta.
\end{equations}
\end{lem}

\begin{proof} Fix $\epsi > 0$. We observe that the family of bounded operators on $L^2_0$
\begin{equation}\label{eq:1u}
(\delta, \xi) \in [0,\delta_\sharp) \times \big\{ \xi \in \R^2 : \ \rho(\xi) \geq \epsi \big\} \  \mapsto \ e^{-i\lr{\xi,x}} \Pi_\delta(\xi)e^{i\lr{\xi,x}}
\end{equation}
is smooth because of \eqref{eq:2a} and \cite[\S VII.1.3, Theorem 1.7]{Ka}.  Therefore, the estimate
\begin{equation}
\Pi_\delta(\xi) \big[ \nabla_{\xi_1} \Pi_\delta(\xi), \nabla_{\xi_2} \Pi_\delta(\xi) \big] = \Pi_0(\xi) \big[ \nabla_{\xi_1} \Pi_0(\xi), \nabla_{\xi_2} \Pi_0(\xi) \big] + \OO_{L^2_\xi}(\delta)
\end{equation}
holds uniformly for $\xi$ in compact subsets of $\R^2$ with $\rho(\xi) \geq \epsi$. Since \eqref{eq:1u} varies periodically with $\xi$, it holds uniformly on $\big\{ \xi \in \R^2 : \ \rho(\xi) \geq \epsi \big\}$. The remainder $\OO_{L^2_\xi}(\delta)$ is an operator of rank at most $2n$ because the leading order terms are of rank $n$. Therefore we can take the trace on both sides and deduce 
\begin{equation}
\Tr\left(\Pi_\delta(\xi) \big[ \nabla_{\xi_1} \Pi_\delta(\xi), \nabla_{\xi_2} \Pi_\delta(\xi) \big] \right) = \Tr\left(\Pi_0(\xi) \big[ \nabla_{\xi_1} \Pi_0(\xi), \nabla_{\xi_2} \Pi_0(\xi) \big]\right) + O(\delta),
\end{equation}
uniformly for $\xi \in \R^2$ with $\rho(\xi) \geq \epsi$. Hence $\Bb_\delta(\xi) = \Bb_0(\xi) + O(\delta)$. Since $\Bb_0(\xi) = 0$, the proof is complete.
\end{proof}

\subsection{Near Dirac momenta}
Fix a Dirac point $(\xi_\star,E_\star)$ of $\Pp_0$. In this section, we estimate $\Bb_\delta$. We first prove spectral estimates at pairs momentum / energy $(\xi,z)$ near $(\xi_\star,E_\star)$. We recall the identity \cite[Proposition 4.5]{FLTW3} -- see \cite[Lemma 2.1]{Dr0} for the version needed here: there exists $\nu_\star \in \C$ with $|\nu_\star| = \nu_F$ such that
\begin{equation}\label{eq:1i}
\forall \eta \in \R^2 \equiv \C, \ \ \  \ \nu_\star \eta = 2 \big \langle \phi_1, (\eta \cdot D_x) \phi_2 \big \rangle_{L^2_{\xi_\star}}, \ \ \ \ D_x \de \dfrac{1}{i} \matrice{\p_{x_1} \\ \p_{x_2}}.
\end{equation}
Introduce the matrix
\begin{equation}
\Mm_\delta(\xi) = \matrice{E_\star + \delta \te_\star & \nu_\star \cdot (\xi-\xi_\star) \\ \ove{\nu_\star  \cdot (\xi-\xi_\star)} & E_\star-\delta\te_\star}, \ \ \ \ \te_\star \de \lr{\phi_1,\Ww\phi_1}_{L^2_{\xi_\star}}.
\end{equation}
For every $\xi \in \R^2$ and $\delta > 0$, the matrix $\Mm_\delta(\xi)$ has two distinct eigenvalues
\begin{equation}
\mu_\delta^\pm(\xi) = E_\star \pm r_\delta(\xi), \ \ \ \ r_\delta(\xi) \de \sqrt{ \te_F^2 \cdot \delta^2+ \nu_F^2 \cdot |\xi-\xi_\star|^2}, \ \ \ \ \te_F \de |\te_\star|.
\end{equation}
 The difference between the two eigenvalues of $\Mm_\delta(\xi)$ is $2r_\delta(\xi)$. Hence, 
\begin{equation}\label{eq:1b}
z \in \p\Dd\big(\mu_\delta^\pm(\xi),r_\delta(\xi)\big) \ \ \Rightarrow \ \ \big|z-\mu_\delta^{\mp}(\xi) \big| \geq r_\delta(\xi).
\end{equation}
The spectral theorem shows that for $z \in \p\Dd\big(\mu_\delta^\pm(\xi),r_\delta(\xi)\big)$, $\Mm_\delta(\xi) - z$ is invertible and
\begin{equation}\label{eq:1d}
\big(z-\Mm_\delta(\xi)\big)^{-1}  = \OO_{\C^2}\left(r_\delta(\xi)^{-1}\right) = \OO_{\C^2}\left( (\delta+|\xi-\xi_\star|)^{-1} \right).
\end{equation}

Introduce the operator
\begin{equation}
\Jj_0(\xi) : L^2_\xi \rightarrow \C^2, \ \ \ \ \Jj_0(\xi)  u = \matrice{ \big\langle e^{i\lr{\xi-\xi_\star,x}} \phi_1, \  u \big\rangle_{L^2_\xi} \\ \big\langle e^{i\lr{\xi-\xi_\star,x}} \phi_2, \ u \big\rangle_{L^2_\xi}}.
\end{equation}

\begin{lem}\label{lem:1d} Assume that $\te_\star \neq 0$. There exist $\delta_0$ and $\epsi_0 > 0$ such that if
\begin{equation}\label{eq:1f}
\delta \in (0,\delta_0), \ \ |\xi - \xi_\star| < \epsi_0, \ \ z \in \p \Dd\big(\mu_\delta^\pm(\xi),r_\delta(\xi)\big)
\end{equation}
then $\Pp_{\delta,+}(\xi)-z : H^2_\xi \rightarrow L^2_\xi$ is invertible and 
\begin{equation}
 \big(z-\Pp_{\delta,+}(\xi) \big)^{-1} = \Jj_0(\xi)^* \cdot \big( z-\Mm_\delta(\xi) \big)^{-1} \cdot \Jj_0(\xi) + \OO_{L^2_\xi}(1).
\end{equation}
\end{lem}

\begin{proof} 1. We proved an analogous statement in \cite[Lemma 4.3]{Dr0} for different values of the parameters $\xi$ and $z$. Here we require $|\xi-\xi_\star| \leq \epsi_0$ and $z \in \p \Dd\big(\mu_\delta^\pm(\xi),r_\delta(\xi)\big)$ instead of $|\xi-\xi_\star| \leq \delta^{1/3}$ and $z \in \Dd\big(E_\star,\te \cdot r_\delta(\xi)\big)$ for some $\te \in (0,1)$. The same strategy works here. Introduce the $\xi$-dependent family of vector spaces
\begin{equation}
\VVV(\xi) = \C \cdot e^{i\blr{\xi-\xi_\star,x}} \phi_1 \oplus \C \cdot e^{i\blr{\xi-\xi_\star,x}} \phi_2 \subset L^2_\xi.
\end{equation}
We split $L^2_\xi$ as $\VVV(\xi) \oplus \VVV(\xi)^\perp$. With respect to this decomposition, we write $\Pp_{\delta,+}(\xi)$ as a block-by-block operator:
\begin{equation}\label{eq:5p}
\Pp_{\delta,+}(\xi) - z = \matrice{A_\delta(\xi)-z & B_\delta(\xi) \\ C_\delta(\xi) & D_\delta(\xi)-z}.
\end{equation}
Below, we use $\lr{\cdot, \cdot}$ instead of $\lr{\cdot, \cdot}_{L^2_\xi}$ to denote the Hermitian product on $L^2_\xi$.

2. Bounds for the operators $B_\delta(\xi)$ and $C_\delta(\xi)$ were obtained in \cite[(4.8)]{Dr0}:
\begin{equation}\label{eq:1g}
 B_\delta(\xi)  = \OO_{\VVV(\xi)^\perp \rightarrow \VVV(\xi)}\big(r_\delta(\xi)\big), \ \ \ \ C_\delta(\xi) = \OO_{\VVV(\xi) \rightarrow \VVV(\xi)^\perp}\big(r_\delta(\xi)\big).
\end{equation}

3. Step 3 in the proof of \cite[Lemma 4.3]{Dr0} applies here. It uses that $D_0(\xi)$ has no eigenvalues near $E_\star$ and that $D_\delta(\xi)-D_0(\xi) = \OO_{L^2_\xi}(\delta)$. It shows that if \eqref{eq:1f} holds then $D_\delta(\xi) - z$ is invertible from $\VVV(\xi)^\perp \cap H^2_\xi$ to $\VVV(\xi)^\perp$ and 
\begin{equations}\label{eq:1h}
\big( D_\delta(\xi) - z \big)^{-1}  = \OO_{\VVV(\xi)^\perp}(1).
\end{equations}

4. We now study $A_\delta(\xi)-z$. This operator acts on the two-dimensional space $\VVV(\xi)$; its matrix in the basis $\big\{ e^{i\blr{\xi-\xi_\star,x}} \phi_1, e^{i\blr{\xi-\xi_\star,x}} \phi_2 \big\}$ is
\begin{equation}\label{eq:3c}
\matrice{ \blr{ e^{i\blr{\xi-\xi_\star,x}} \phi_1, (\Pp_\delta(\xi)-z)e^{i\blr{\xi-\xi_\star,x}} \phi_1 } &  \blr{ e^{i\blr{\xi-\xi_\star,x}} \phi_1, (\Pp_\delta(\xi)-z)e^{i\blr{\xi-\xi_\star,x}} \phi_2 } \\
\blr{ e^{i\blr{\xi-\xi_\star,x}} \phi_2, (\Pp_\delta(\xi)-z)e^{i\blr{\xi-\xi_\star,x}} \phi_1 } & \blr{ e^{i\blr{\xi-\xi_\star,x}} \phi_2, (\Pp_\delta(\xi)-z)e^{i\blr{\xi-\xi_\star,x}} \phi_2 }}.
\end{equation}
As in \cite[Step 4, Lemma 4.3]{Dr0} the matrix elements in \eqref{eq:3c} are
\begin{equations}
\blr{\phi_j, \big( \Pp_\delta(\xi)-z\big) \phi_k} =  \big( E_\star - |\xi-\xi_\star|^2 -z \big) \delta_{jk} + \blr{\phi_j, \big( \delta \Ww  +  2 (\xi-\xi_\star) \cdot D_x\big) \phi_k}.
\end{equations}
Because of \eqref{eq:1i}, $\blr{\phi_2,  2 (\xi-\xi_\star) \cdot D_x \phi_1} = \nu_\star(\xi-\xi_\star)$; \cite[Lemma 2.1]{Dr0} shows that $\blr{\phi_j,  2 (\xi-\xi_\star) \cdot D_x \phi_j}$ vanishes. Moreover, \cite[Lemma 7.3]{Dr0} shows that $\lr{\phi_2, \Ww \phi_1} = \lr{\phi_1, \Ww \phi_2} = 0$ and $\lr{\phi_1, \Ww \phi_1} = \te_\star = -\lr{\phi_2, \Ww \phi_2}$. We deduce that the matrix \eqref{eq:3c} is equal to $\Mm_\delta(\xi)-z + \OO_{\C^2}(\xi-\xi_\star)^2$. Using a Neumann series argument based on \eqref{eq:3c}, when \eqref{eq:1f} holds, $A_\delta(\xi)-z$ is invertible; and
\begin{equations}\label{eq:4q}
\big( A_\delta(\xi)-z \big)^{-1} =  \Ii_0(\xi)^* \cdot \big( \Mm_\delta(\xi)-z \big)^{-1} \cdot \Ii_0(\xi) + \OO_{\VVV(\xi)}\left( \dfrac{|\xi-\xi_\star|^2}{r_\delta(\xi)^2}\right)
\\
= \Ii_0(\xi)^* \cdot \big( \Mm_\delta(\xi)-z \big)^{-1} \cdot \Ii_0(\xi) + \OO_{\VVV(\xi)}(1).
\end{equations}
Above, $\Ii_0(\xi) : \VVV(\xi) \rightarrow \C^2$ is the coordinate map.
Because of \eqref{eq:1d}, we also get
\begin{equation}\label{eq:4r}
\big( A_\delta(\xi)-z \big)^{-1}  = \OO_{\VVV(\xi)} \left(r_\delta(\xi)^{-1} \right).
\end{equation}

5. Schur's lemma allows to invert block-by-block operators of the form \eqref{eq:5p} under certain conditions on the blocks; see \cite[Lemma 4.1]{DFW} for the version needed here. We checked that $D_\delta(\xi) - z : \VVV(\xi) \rightarrow \VVV(\xi)$ is invertible. 
It remains to check that:
\begin{equations}\label{eq:4o}
A_\delta(\xi) - z - B_\delta(\xi) \cdot \big( D_\delta(\xi)-z \big)^{-1} \cdot C_\delta(\xi) : \VVV(\xi) \rightarrow \VVV(\xi) \ \text{ is invertible}.
\end{equations}
We observe that because of \eqref{eq:1g} and \eqref{eq:1h},
\begin{equation}
B_\delta(\xi) \cdot \big( D_\delta(\xi)-z \big)^{-1} \cdot C_\delta(\xi) = \OO_{\VVV(\xi)}\big(r_\delta(\xi)^2 \big).
\end{equation}
Therefore a Neumann series argument based on \eqref{eq:4r} shows that \eqref{eq:4o} holds. Thanks to \eqref{eq:4q}, it also shows that the inverse is equal to
\begin{equation}
\big( A_\delta(\xi) - z \big)^{-1} + \OO_{\VVV(\xi)}(1)
= \Ii_0(\xi)^* \cdot \big( \Mm_\delta(\xi)-z \big)^{-1} \cdot \Ii_0(\xi) + \OO_{\VVV(\xi)}(1).
\end{equation}

We apply Schur's lemma. From \eqref{eq:5p}, we obtain that $\Pp_\delta(\xi) - z : H^2_\xi \rightarrow L^2_\xi$ is invertible when \eqref{eq:1f} holds; and moreover
\begin{equations}
\big(\Pp_\delta(\xi) - z\big)^{-1} = \matrice{\Ii_0(\xi)^* \cdot \big( \Mm_\delta(\xi)-z \big)^{-1} \cdot \Ii_0(\xi) & 0 \\ 0 & 0} + \OO_{L^2_\xi}(1) 
\\
=\Jj_0(\xi)^* \cdot \big( \Mm_\delta(\xi)-z \big)^{-1} \cdot \Jj_0(\xi) + \OO_{L^2_\xi}(1).
\end{equations}
This completes the proof.
\end{proof}

\begin{lem}\label{lem:1b} If \eqref{eq:2a} holds, then $\te_\star \neq 0$ and $\delta_\sharp > 0$ for a generic choice of $\Ww$.
\end{lem}

\begin{proof}  1. Recall that $\te_\star = \lr{\phi_1,\Ww\phi_1}_{L^2_{\xi_\star}} $ and observe that
\begin{equation}\label{eq:1v}
\te_\star =  \lr{\phi_1,\Aa D_x\phi_1}_{L^2_{\xi_\star}} + \lr{D_x\phi_1,\Aa \phi_1}_{L^2_{\xi_\star}}  = \blr{\ove{\phi_1} \cdot D_x\phi_1-(D_x\ove{\phi_1}) \cdot \phi_1,\Aa }_{L^2_{\xi_\star}}.
\end{equation}
Because of the unique continuation principle for elliptic problems -- see \cite[Theorem 17.2.6]{Ho} -- $\phi_1$ cannot vanish on an open set. We deduce that if 
\begin{equation}
\ove{\phi_1} D_x\phi_1-D_x\ove{\phi_1}\phi_1 = \ove{\phi_1}^2 D_x \left(\frac{\phi_1}{\ove{\phi_1}}\right)
\end{equation}
vanishes uniformly, then $\phi_1$ and $\ove{\phi_1}$ are linearly dependent. This is impossible because $\ove{\phi_1} \in L^2_{-\xi_\star}$ and $L^2_{\xi_\star} \cap L^2_{-\xi_\star} = \{0\}$. We deduce that the condition $\te_\star \neq 0$ is equivalent to requiring that $\Aa$ does not lie in the hyperplane normal to $\ove{\phi_1} \cdot D_x\phi_1-(D_x\ove{\phi_1}) \cdot \phi_1$. This is a generic condition.

2. Define $\gamma_\delta^\pm(\xi) = \p \Dd\big(\mu_\delta^\pm(\xi),r_\delta(\xi)\big)$. Lemma \ref{lem:1d} implies that
\begin{equations}
\Pi_\delta^\pm(\xi) \de \dfrac{1}{2\pi i}\oint_{\gamma_\delta^\pm(\xi)} \big(z - \Pp_{\delta,+}(\xi) \big)^{-1} \cdot dz 
\\
= \dfrac{1}{2\pi i}\oint_{\gamma_\delta^\pm(\xi)} \Jj_0(\xi)^*  \cdot \big( z- \Mm_\delta(\xi) \big)^{-1} \cdot  \Jj_0(\xi) \cdot dz  + \OO_{L^2_\xi}\big(r_\delta(\xi)^2\big) \de \pi_\delta^\pm(\xi) + \OO_{L^2_\xi}\big(r_\delta(\xi)^2\big).
\end{equations}
Because of the spectral theorem, $\Pi_\delta^\pm(\xi)$ is a projector. If $f_1, f_2$ are normalized elements in the range of $\Pi_\delta(\xi)$ then
\begin{equation}
f_1 = \Pi_\delta^\pm(\xi) f_1 = \pi_\delta(\xi) f_1 + O_{L^2_\xi}\big(r_\delta(\xi)^2\big).
\end{equation}
Since $\pi_\delta(\xi)$ is a rank-one projector -- see \eqref{eq:1b} -- $f_1$ and $f_2$ cannot be orthogonal. We deduce that $\Pi_\delta^\pm(\xi)$ has rank one.  In other words  $\Pp_{\delta,+}(\xi)$ has precisely one eigenvalue in each disk $\Dd\big(\mu_\delta^\pm(\xi),r_\delta(\xi)\big)$ for $(\xi,\delta)$ close enough to $(\xi_\star,0)$. Because of \cite[Appendix A.1]{FW2}, these two eigenvalues must be $\lambda_{\delta,n}(\xi)$ and $\lambda_{\delta,n+1}(\xi)$. 
We deduce that
\begin{equations}
 \lambda_{\delta,n}(\xi) \cdot \Pi_\delta^-(\xi)  = \dfrac{1}{2\pi i}\oint_{\gamma_\delta^-(\xi)} z \cdot \big(z - \Pp_{\delta,+}(\xi) \big)^{-1} \cdot dz 
\\
 =   \Jj_0(\xi)^* \cdot \dfrac{1}{2\pi i}\oint_{\gamma_\delta^-(\xi)}  z \cdot \big( z- \Mm_\delta(\xi) \big)^{-1}  \cdot dz\cdot \Jj_0(\xi) + \OO_{L^2_\xi}\big(r_\delta(\xi)^2\big)
 \\
 =\mu_\delta^-(\xi) \cdot \Jj_0(\xi)^* \pi_\delta(\xi) \Jj_0(\xi) + \OO_{L^2_\xi}\big(r_\delta(\xi)^2\big).
\end{equations}
A similar identity holds for $\lambda_{\delta,n}(\xi) \cdot \Pi_\delta^+(\xi)$. Taking the trace, we deduce that
\begin{equations}\label{eq:8c}
\lambda_{\delta,n}(\xi) = \mu_\delta^-(\xi) + O\big( r_\delta(\xi)^2\big) = E_\star - r_\delta(\xi) + O\big( r_\delta(\xi)^2\big), 
\\
\lambda_{\delta,n+1}(\xi) = \mu_\delta^+(\xi) + O\big( r_\delta(\xi)^2\big) = E_\star + r_\delta(\xi) + O\big( r_\delta(\xi)^2\big).
\end{equations}

3. Assume that $\delta_\sharp = 0$. Then for any $k \in \N$, there exist $\xi_k \in \Ll^*$ and $0 < \delta_k \rightarrow 0$ as $k \rightarrow \infty$, with $\lambda_{\delta_k,n}(\xi_k) = \lambda_{\delta_k,n+1}(\xi_k)$. 
After passing to a subsequence, we can assume that $\xi_k$ converges to a point$\rightarrow \xi_\infty$. Because of \cite[Appendix A.1]{FW2}, $\lambda_{\delta_k,n}(\xi_k) \rightarrow \lambda_{0,n}(\xi_\infty)$ and $\lambda_{\delta_k,n+1}(\xi_k) \rightarrow \lambda_{0,n+1}(\xi_\infty)$. It follows that $\lambda_{0,n}(\xi_\infty) = \lambda_{0,n+1}(\xi_\infty)$. We deduce from \eqref{eq:2a} that $\xi_\infty \in \{\xi_\star^A,\xi_\star^B\}$; \eqref{eq:1b} and  \eqref{eq:8c} yield
\begin{equation}
E_\star + r_{\delta_k}(\xi_k) + O\big(r_{\delta_k}(\xi)^2\big) = E_\star - r_{\delta_k}(\xi_k) + O\big(r_{\delta_k}(\xi)^2\big).
\end{equation}
This is not possible unless $\delta_k = 0$ for $k$ large enough, which contradicts $\delta_k > 0$. We conclude that $\delta_\sharp > 0$.
\end{proof}

Let $\xi \in \R^2 \mapsto \BB_\delta(\xi)$ be the trace of the Berry curvature associated to the line bundle with fiber $\ker_{\C^2}\big( \Mm_\delta(\xi) - \mu_\delta^-(\xi) \big)$ over $\R^2$.

\begin{lem}\label{lem:1c} There exist $\delta_0 > 0$ and $\epsi_0 > 0$ such that
\begin{equation}
\delta \in (0,\delta_0), \ \ |\xi - \xi_\star| < \epsi_0 \ \ \Rightarrow \ \ \Bb_\delta(\xi) =  \BB_\delta(\xi) + O\big(r_\delta(\xi)^{-1}\big).
\end{equation}
\end{lem}

\begin{proof} 1. Let $Q_\delta(\xi) : L^2_\xi \rightarrow L^2_\xi$ the projector on 
\begin{equation}
\bigoplus_{j=1}^{n-1} \ker_{L^2_\xi}\big(\Pp_{\delta,+}(\xi)-\lambda_{\delta,j}(\xi) \big).
\end{equation}
The eigenvalues $\lambda_{0,1}(\xi_\star), \dots, \lambda_{0,n-1}(\xi_\star)$ of $\Pp_0(\xi_\star)$ are separated from the rest of the spectrum of $\Pp_0(\xi_\star)$ because $E_\star = \lambda_{0,n}(\xi_\star) = \lambda_{0,n+1}(\xi_\star)$ has multiplicity precisely $2$. Because of \cite[\S VIII.1.3 Theorem 1.7]{Ka}, the family $(\delta,\xi) \mapsto Q_\delta(\xi)$ is smooth on a neighborhood of $(0,\xi_\star)$. In particular, under these conditions, 
\begin{equation}
\Tr\Big( Q_\delta(\xi) \big[ \nabla_{\xi_1}Q_\delta(\xi), \nabla_{\xi_2}Q_\delta(\xi)\big] \Big) = O(1).
\end{equation}

For $\delta > 0$, let $\Pi_\delta^-(\xi)$ be the projector on $\ker_{L^2_\xi}\big(\Pp_{\delta,+}(\xi)-\lambda_{\delta,n}(\xi) \big)$. Because of \eqref{eq:8c}, for $(\delta,\xi)$ near $(0,\xi_\star)$, $\lambda_{\delta,n}(\xi)$ is a simple eigenvalue of $\Pp_{\delta,+}(\xi)$. Therefore for $(\delta,\xi)$ near $(0,\xi_\star)$ the projector $\Pi_\delta(\xi)$ splits orthogonally as
\begin{equation}\label{eq:1w}
\Pi_\delta(\xi) = \Pi_\delta^-(\xi) + Q_\delta(\xi).
\end{equation}

We recall that $\Bb_\delta(\xi)$ is gauge independent -- i.e. it does not depend on the choice of frame in \eqref{eq:3h}. Pick a frame in \eqref{eq:3h} associated to the orthogonal decomposition \eqref{eq:1w}. The associated Berry connection and curvature split accordingly to components for $\Pi_\delta^-(\xi)$ and $Q_\delta(\xi)$. In other words, the curvature endomorphism is a two-form valued  diagonal by block matrix, with respect to the decomposition \eqref{eq:1w}.  Therefore,
\begin{equations}\label{eq:1j}
\Bb_\delta(\xi) = \Tr\Big( Q_\delta(\xi) \big[ \nabla_{\xi_1}Q_\delta(\xi), \nabla_{\xi_2}Q_\delta(\xi)\big] \Big) + \Tr\Big( \Pi_\delta^-(\xi) \big[ \nabla_{\xi_1}\Pi_\delta^-(\xi), \nabla_{\xi_2}\Pi_\delta^-(\xi)\big] \Big) 
\\
= \Tr\Big( \Pi_\delta^-(\xi) \big[ \nabla_{\xi_1}\Pi_\delta^-(\xi), \nabla_{\xi_2}\Pi_\delta^-(\xi)\big] \Big) + O(1).
\end{equations}

2. Recall that $\gamma_\delta^-(\xi) = \p\Dd\big(\mu_\delta^-(\xi),r_\delta(\xi)\big)$ oriented clockwise. When $(\xi,\delta)$ is sufficiently close to $(\xi_\star,0)$, we saw in Step 2 of the proof of Lemma \ref{lem:1b} that $\Pp_{\delta,+}(\xi)$ has a unique eigenvalue in $\Dd\big(\mu_\delta^\pm(\xi), r_\delta(\xi)\big)$, which is $\lambda_{\delta,n}(\xi)$. The Cauchy formula yields
\begin{equations}\label{eq:3a}
\Pi_\delta^-(\xi) = \dfrac{1}{2\pi i}\oint_{\gamma^-_\delta(\xi)} \big(z - \Pp_{\delta,+}(\xi) \big)^{-1} dz .
\end{equations}
Let $\pi^-_\delta(\xi)$ be the projector on the eigenvalue $\mu_\delta^-(\xi)$ of $\Mm_\delta(\xi)$.
We use Lemma \ref{lem:1d} and that $\gamma^-_\delta(\xi)$ has length $O\big(r_\delta(\xi)\big)$ to get 
\begin{equations}
\Pi_\delta^-(\xi) = \Jj_0(\xi)^* \cdot \dfrac{1}{2\pi i}\oint_{\gamma^-_\delta(\xi)} \big(z -  \Mm_\delta(\xi) \big)^{-1} dz \cdot \Jj_0(\xi) + \OO_{L^2_\xi}\big(r_\delta(\xi)\big) \\
 = \Jj_0(\xi)^* \cdot \pi^-_\delta(\xi) \cdot \Jj_0(\xi) + \OO_{L^2_\xi}\big(r_\delta(\xi)\big).
\end{equations}

3. We study $\nabla_{\xi_1} \Pi^-_\delta(\xi)$ -- defined in \eqref{eq:3i}:
\begin{equations}\label{eq:2k}
\nabla_{\xi_1} \Pi^-_\delta(\xi) \de e^{i\lr{\xi,x}} \cdot \dd{\left(e^{-i\lr{\xi,x}} \Pi^-_\delta(\xi) e^{i\lr{\xi,x}}\right) }{\xi_1} \cdot e^{-i\lr{\xi,x}} \ : \  L^2_\xi \rightarrow L^2_\xi.
\end{equations} 
The projector $e^{-i\lr{\xi,x}} \Pi^-_\delta(\xi) e^{i\lr{\xi,x}}$ is associated to $e^{-i\lr{\xi,x}} \Pp_\delta(\xi) e^{i\lr{\xi,x}}$ instead of $\Pp_\delta(\xi)$. The same Cauchy formula as \eqref{eq:3a} gives
\begin{equation}\label{eq:2j}
e^{-i\lr{\xi,x}} \Pi^-_\delta(\xi) e^{i\lr{\xi,x}} = \dfrac{1}{2\pi i} \oint_{\gamma^-_\delta(\xi)} \left( z - e^{-i\lr{\xi,x}} \Pp_{\delta,+}(\xi) e^{i\lr{\xi,x}} \right)^{-1} dz.
\end{equation}
Observe that $e^{-i\lr{\xi,x}} \Pp_\delta(\xi) e^{i\lr{\xi,x}} = (D_x+\xi)^2 + \Vv + \delta \Ww$ . Hence, 
\begin{equation}
\dd{ \left(e^{-i\lr{\xi,x}} \Pp_{\delta,+}(\xi) e^{i\lr{\xi,x}} \right)}{\xi_1} = 2(D_{x_1}+\xi_1) = 2e^{-i\lr{\xi,x}} D_{x_1} e^{i\lr{\xi,x}}.
\end{equation}

We use \cite[Lemma A.6]{Dr2} -- a result to differentiate Cauchy integrals when the contour depends on the parameter -- to differentiate \eqref{eq:2j} w.r.t. $\xi_1$. We get
\begin{equations}
\dfrac{1}{2\pi i} \oint_{\gamma^-_\delta(\xi)}  \left( z - e^{-i\lr{\xi,x}} \Pp_{\delta,+}(\xi) e^{i\lr{\xi,x}} \right)^{-1}  \dd{ \left(e^{-i\lr{\xi,x}} \Pp_{\delta,+}(\xi) e^{i\lr{\xi,x}} \right)}{\xi_1} \left( z - e^{-i\lr{\xi,x}} \Pp_{\delta,+}(\xi) e^{i\lr{\xi,x}} \right)^{-1} dz
\\
= \dfrac{1}{2\pi i} \oint_{\gamma^-_\delta(\xi)}   \left( z - e^{-i\lr{\xi,x}} \Pp_{\delta,+}(\xi) e^{i\lr{\xi,x}} \right)^{-1}  \cdot 2 e^{-i\lr{\xi,x}} D_{x_1} e^{i\lr{\xi,x}} \cdot \left( z - e^{-i\lr{\xi,x}} \Pp_{\delta,+}(\xi) e^{i\lr{\xi,x}} \right)^{-1} dz
\\
= e^{-i \lr{\xi,x}} \cdot \dfrac{1}{2\pi i} \oint_{\gamma^-_\delta(\xi)}   \big( z - \Pp_{\delta,+}(\xi)\big)^{-1}  \cdot 2D_{x_1} \cdot \big( z -  \Pp_{\delta,+}(\xi)  \big)^{-1} dz \cdot e^{i \lr{\xi,x}}.
\end{equations}
We deduce from \eqref{eq:2k} that
\begin{equations}\label{eq:3b}
\nabla_{\xi_1} \Pi^-_\delta(\xi) = \dfrac{1}{2\pi i} \oint_{\gamma^-_\delta(\xi)}   \big( z - \Pp_{\delta,+}(\xi)\big)^{-1}  \cdot 2D_{x_1} \cdot \big( z -  \Pp_{\delta,+}(\xi)  \big)^{-1} dz.
\end{equations} 

We recall that $\big(z- \Mm_\delta(\xi) \big)^{-1} = \OO_{\C^2}\big(r_\delta(\xi)^{-1}\big)$ when $z \in \gamma^-_\delta(\xi)$. 
Because of Lemma \ref{lem:1d}, $(z - \Pp_{\delta,+}(\xi))^{-1} = \OO_{L^2_\xi}\big(r_\delta(\xi)^{-1}\big)$ when $z \in \gamma^-_\delta(\xi)$. Since the contour $\gamma^-_\delta(\xi)$ has length $\OO\big(r_\delta(\xi)\big)$, we deduce from \eqref{eq:3b} that $\nabla_{\xi_1} \Pi_\delta^-(\xi) = \OO_{L^2_\xi}\big(r_\delta(\xi)^{-1}\big)$. Moreover, Lemma \ref{lem:1d} combined with \eqref{eq:3b} shows that  $\nabla_{\xi_1} \Pi_\delta^-(\xi)$ equals
\begin{equations}
\Jj_0(\xi)^* \cdot \dfrac{1}{2\pi i}\oint_{\gamma^-_\delta(\xi)} \big(z- \Mm_\delta(\xi) \big)^{-1} \cdot \Jj_0(\xi)    2D_{x_1}   \Jj_0(\xi)^* \cdot \big(z- \Mm_\delta(\xi) \big)^{-1} dz \cdot \Jj_0(\xi) + \OO_{L^2_\xi}(1).
\end{equations}
As in Step 4 in the proof of Lemma \ref{lem:1d}, 
\begin{equation}
\Jj_0(\xi)   2D_{x_1}   \Jj_0(\xi)^* =   \matrice{0 & \nu_\star \\ \ove{\nu_\star} & 0} + \OO_{\C^2}(\xi-\xi_\star) = \dd{\Mm_\delta(\xi)}{\xi_1}+ \OO_{\C^2}(\xi-\xi_\star).
\end{equation}
We use $\big(z- \Mm_\delta(\xi) \big)^{-1} = \OO_{\C^2}\big(r_\delta(\xi)^{-1}\big)$ when $z \in \gamma^-_\delta(\xi)$; $\gamma^-_\delta(\xi)$ has length $O\big(r_\delta(\xi)\big)$; $|\xi-\xi_\star| \cdot r_\delta(\xi)^{-1} = O(1)$; to deduce that $\nabla_{\xi_1} \Pi_\delta^-(\xi)$ equals, modulo $OO_{L^2_\xi}(1)$,
\begin{equations}
 \Jj_0(\xi)^* \cdot \dfrac{1}{2\pi i}\oint_{\gamma^-_\delta(\xi)} \big(z- \Mm_\delta(\xi) \big)^{-1} \dd{\Mm_\delta(\xi)}{\xi_1} \big(z- \Mm_\delta(\xi) \big)^{-1} dz \cdot \Jj_0(\xi) 
\\= \Jj_0(\xi)^* \cdot \dd{}{\xi_1}\left( \dfrac{1}{2\pi i}\oint_{\gamma^-_\delta(\xi)} \big(z- \Mm_\delta(\xi) \big)^{-1}  dz \right) \cdot \Jj_0(\xi) + = \Jj_0(\xi)^* \cdot \dd{\pi^-_\delta(\xi)}{\xi_1} \cdot \Jj_0(\xi).
\end{equations}
A similar calculation leads to
\begin{equation}
\dd{\Pi_\delta^-(\xi)}{\xi_2} = \OO_{L^2_\xi}\big(r_\delta(\xi)^{-1}\big), \ \ \ \ \dd{\Pi_\delta^-(\xi)}{\xi_2} = \Jj_0(\xi)^* \cdot \dd{\pi^-_\delta(\xi)}{\xi_2} \cdot \Jj_0(\xi) + \OO_{L^2_\xi}(1).
\end{equation}

4. We conclude that
\begin{equations}\label{eq:8a}
\Pi_\delta^-(\xi) \big[ \nabla_{\xi_1}\Pi_\delta^-(\xi), \nabla_{\xi_2}\Pi_\delta^-(\xi)\big]
\\
 = \Jj_0(\xi)^* \cdot \pi^-_\delta(\xi) \big[ \p_{\xi_1}\pi^-_\delta(\xi), \p_{\xi_2}\pi^-_\delta(\xi)\big] \cdot \Jj_0(\xi) +  \OO_{L^2_\xi}\big(r_\delta(\xi)^{-1}\big).
\end{equations}
Since the terms on both sides have rank at most $1$, the remainder term has rank at most $2$ and we can take the trace of \eqref{eq:8a}, without changing the magnitude of the remainder. This yields
\begin{equation}
\Tr\Big( \Pi_\delta^-(\xi) \big[ \nabla_{\xi_1}\Pi_\delta^-(\xi), \nabla_{\xi_2}\Pi_\delta^-(\xi)\big] \Big) = \Tr\Big( \pi^-_\delta(\xi) \big[ \nabla_{\xi_1}\pi^-_\delta(\xi), \nabla_{\xi_2}\pi^-_\delta(\xi)\big] \Big) +  O\big(r_\delta(\xi)^{-1}\big).
\end{equation}
Above we used the cyclicity of the trace and the formula $\Jj_0(\xi) \Jj_0(\xi)^* = \Id_{\C^2}$ to get rid of the terms $\Jj_0(\xi)$ and $\Jj_0(\xi)^*$. The proof is complete thanks to \eqref{eq:1j}.
\end{proof}

Observe that because $r_\delta(\xi)$ is bounded below by $\nu_F |\xi-\xi_\star|$, for every $\epsi > 0$, 
\begin{equation}\label{eq:1q}
\int_{\Dd(\xi_\star,\epsi)} \dfrac{d\xi}{r_\delta(\xi)} \leq \int_0^\epsi \dfrac{2\pi r \cdot dr}{\nu_F r}  = \dfrac{2\pi \epsi}{\nu_F}.
\end{equation}
We deduce from Lemma \ref{lem:1c} and \eqref{eq:1q} that for every $\epsi \in (0,\epsi_0)$ and $\delta \in (0,\delta_0)$, there exists a constant $C_0$ such that
\begin{equation}
\left|\dfrac{i}{2\pi}  \int_{\Dd(\xi_\star,\epsi)} \Bb_\delta(\xi) d\xi - \dfrac{i}{2\pi} \int_{\Dd(\xi_\star,\epsi)} \BB_\delta(\xi) d\xi \right| \leq C_0\epsi.
\end{equation}
Fix $\epsi_1$ such that $\epsi_1 \leq \epsi_0$ and $4C_0\epsi_1 \leq 1$. Then
\begin{equation}\label{eq:3d}
\left|\dfrac{i}{2\pi}  \int_{\Dd(\xi_\star,\epsi_1)} \Bb_\delta(\xi) d\xi - \dfrac{i}{2\pi} \int_{\Dd(\xi_\star,\epsi_1)} \BB_\delta(\xi) d\xi \right| \leq \dfrac{1}{4}.
\end{equation}
The next lemma computes explicitly $\BB_\delta(\xi)$ and the associated Chern number.

\begin{lem}\label{lem:1e} Let $\epsi_1$ be the number fixed above. As $\delta \rightarrow 0$,
\begin{equation}
\dfrac{i}{2\pi} \int_{\Dd(\xi_\star,\epsi_1)} \BB_\delta(\xi) d\xi = -\dfrac{1}{2} \cdot \sgn(\te_\star) + O(\delta).
\end{equation}
\end{lem}

\begin{proof} 1. We first assume that $\te_\star > 0$. Observe that 
\begin{equation}\label{eq:1l}
\Mm_\delta\big( \Phi_\delta(\xi) \big) = E_\star + \delta \te_\star \cdot \MM(\xi)
, \ \ \ \ \MM(\xi) \de \matrice{1 & \xi \\ \ove{\xi} & -1}, \ \ \ \ \Phi_\delta(\xi) \de \xi_\star+ \dfrac{\delta\te_\star  \cdot \ove{\nu_\star}\xi}{\nu_F^2}. \ 
\end{equation}
Above $\C$ is canonically identified with $\R^2$.

2. We compute the Berry curvature $\BB(\xi)$ associated to the negative energy eigenbundle for $\MM(\xi)$, using \cite[(23)]{FC}:
\begin{equations}\label{eq:1m}
\BB(\xi) = \dfrac{i}{2\big(1+|\xi|^2\big)^{3/2}}  \matrice{\xi_1 \\ -\xi_2 \\ 1} \cdot \left( \dd{}{\xi_1} \matrice{\xi_1 \\ -\xi_2 \\ 1} \wedge \dd{}{\xi_2} \matrice{\xi_1 \\ -\xi_2 \\ 1} \right) 
\\
 = \dfrac{i}{2\big(1+|\xi|^2\big)^{3/2}} \matrice{\xi_1 \\ \xi_2 \\ 1} \cdot \left( \matrice{1 \\ 0 \\ 0} \wedge \matrice{0 \\ -1 \\ 0} \right) = \dfrac{i}{2\big(1+|\xi|^2\big)^{3/2}}.
\end{equations}

3. Because of \eqref{eq:1l}, $\BB(\xi) d\xi = \Phi_\delta^* \big(\BB_\delta(\xi)d\xi \big)$. Moreover,
\begin{equation}
\Phi_\delta\left( \Dd\left(0,\frac{\nu_F \epsi_1}{\te_F \delta} \right) \right) = \Dd\big(\xi_\star,\epsi_1 \big).
\end{equation}
It follows that
\begin{equations}
\dfrac{i}{2\pi} \int_{\Dd(\xi_\star,\epsi_1)} \BB_\delta(\xi) d\xi 
= \dfrac{i}{2\pi} \int_{\Dd\big(0,\frac{\nu_F \epsi_1}{\te_F \delta} \big)} \Phi_\delta^* \big(\BB_\delta(\xi) d\xi\big) = \dfrac{i}{2\pi} \int_{\Dd\big(0,\frac{\nu_F \epsi_1}{\te_F \delta} \big)} \BB(\xi) d\xi.
\end{equations}
We now use the formula \eqref{eq:1m} to compute this integral: we have
\begin{equations}\label{eq:8b}
\dfrac{i}{2\pi} \int_{\Dd\big(0,\frac{\nu_F \epsi_1}{\te_F \delta} \big)} \BB(\xi) d\xi = -\dfrac{1}{4\pi} \int_{\Dd\big(0,\frac{\nu_F \epsi_1}{\te_F \delta} \big)} \dfrac{d\xi}{(1+|\xi|^2)^{3/2}}
 =  -\dfrac{1}{2} + O(\delta).
\end{equations}
In the last equality we used that $\epsi_1$ is a fixed constant.

4. We now deal with the case $\te_\star < 0$. In this case, the value \eqref{eq:8b} corresponds to the positive energy eigenbundle  of $\Mm_\delta(\xi_\star + \delta \te_\star \nu_\star^{-1} \xi)$. The positive and negative eigenbundles direct sum to the trivial bundle $\R^2 \times \C^2$, whose total Berry curvature vanishes. We deduce that when $\te_\star < 0$ and $\delta$ goes to zero,
\begin{equation}
\dfrac{i}{2\pi}\int_{\Dd(\xi_\star,\epsi_1)} \BB_\delta(\xi) d\xi = -\dfrac{i}{2\pi} \int_{\Dd\big(0,\frac{\nu_F \epsi_1}{\te_F \delta} \big)} \BB(\xi) d\xi  = \dfrac{1}{2} + O(\delta).
\end{equation}
This completes the proof. \end{proof}

\subsection{Proof of Theorem \ref{thm:1}} We are now ready to prove Theorem \ref{thm:1}. Fix $\epsi_1$ as in \eqref{eq:3d}. Because of the definition of the first Chern class \eqref{eq:2n},
\begin{equation}\label{eq:1s}
c_1(\Ee_{\delta,+}) = \dfrac{i}{2\pi} \int_{\Ll^*} \Bb_\delta(\xi) d\xi = \dfrac{i}{2\pi} \int_{\substack{\xi \in \Ll^* \\ \rho(\xi) \geq \epsi_1}} \Bb_\delta(\xi) d\xi + \dfrac{i}{2\pi}\sum_{J=A,B} \int_{\Dd(\xi_\star^J, \epsi_1)} \Bb_\delta(\xi) d\xi. \ \
\end{equation}
Because of Lemma \ref{lem:1a}, the first integral is $O(\delta)$. The sum in \eqref{eq:1s} reduces to integrals of traces of Berry curvatures $\BB_\delta(\xi)$ associated to low-energy eigenbundles $\Mm_\delta(\xi)$, modulo an error term that is at most $2 \cdot 1/4 = 1/2$ because of \eqref{eq:3d}. Lemma \ref{lem:1e} computes these integrals and shows that their sum equals $-\sgn(\te_\star) + O(\delta)$. We end up with:
\begin{equation}
\big| c_1(\Ee_{\delta,+}) +  \sgn(\te_\star) \big| \leq \dfrac{1}{2} +  O(\delta).
\end{equation}
Making $\delta \rightarrow 0$ and using that $c_1(\Ee_{\delta,+})$ and $\sgn(\te_\star)$ are both integers, we conclude that $c_1(\Ee_{+,\delta}) = - \sgn(\te_\star)$.

We can go from $\Pp_{\delta, +}$ to $\Pp_{\delta,-}$ by simply switching $\Ww$ to $-\Ww$. This changes $\te_\star$ to $-\te_\star$. Therefore $c_1(\Ee_{-,\delta}) = \sgn(\te_\star)$. This completes the proof of Theorem \ref{thm:1}.

\section{Proof of Theorem \ref{thm:2}}

\subsection{The edge problem}\label{sec:3.1} We review the definition of the edge operator $\Pp_\delta$ introduced in \cite[\S1.7]{Dr0}. This operator models interface effect between two materials -- described respectively by $\Pp_{\delta,+}$ and $\Pp_{\delta,-}$ -- along a rational edge $\R v$, $v \in \Lambda$. Write $v = a_1 v_1 + a_2 v_2$ with $a_1, a_2 \in \Z$ relatively prime and set
\begin{equations}\label{eq:3w}
v' \de b_1 v_1 + b_2 v_2, \ \ \ \ a_1 b_2 - a_2 b_1 = 1, \ \ b_1, \ b_2 \in \Z, \\
k \de b_2 k_1 - b_1 k_2, \ \ \ \ k' \de -a_2 k_1 + a_1 k_2.
\end{equations} 
The operator $\Pp_\delta$ is $-\Delta_{\R^2} + \Vv + \delta \cdot  \kappa_\delta \cdot \Ww$, 
where the function $\kappa_\delta \in C^\infty\big(\R^2,\R\big)$ is a domain wall across $\R v$:
\begin{equation}\label{eq:3v}
\kappa_\delta(x) = \kappa(\delta \blr{k',x}), \ \ \ \ \exists L > 0, \ \  \kappa(t) = \systeme{ -1 \text{ when } x \leq -L,
\\
1 \ \ \text{ when } \ \ x \geq L.}
\end{equation}

Since $\lr{k',v} = 0$, the operator $\Pp_\delta$ is periodic w.r.t. $\Z v$ (though it is not periodic w.r.t. $\Lambda$). We denote by $\Pp_\delta[\zeta]$ the operator formally equal to $\Pp_\delta$, but acting on
\begin{equation}\label{eq:4n}
L^2[\zeta] \de \left\{ u \in L^2_\loc(\R^2,\C), \ u(x+v) = e^{i\zeta} u(x), \ \int_{\R^2/\Z v} |u(x)|^2 dx   < \infty\right\}.
\end{equation}
The bulk operators $\Pp_{\delta,\pm}[\zeta]$ prescribe the essential spectrum of $\Pp_{\delta}[\zeta]$:
\begin{equation}
\Sigma_\ess\big( \Pp_{\delta}[\zeta] \big) = \Sigma_\ess\big( \Pp_{\delta,+}[\zeta] \big) \cup \Sigma_\ess\big( \Pp_{\delta,-}[\zeta] \big).
\end{equation}
When \eqref{eq:2b} is satisfied, the operator $\Pp_\delta[\zeta]$ has an $L^2[\zeta]$-gap, containing the energy level $E_\flat(\zeta)$. The edge index $\NN$ of $\Pp_\delta$ in this gap is defined as the signed number of eigenvalues of the family $\Pp_\delta[\zeta] - E_\flat(\zeta)$ crossing $E_\star$ downward as $\zeta$ spans $[0,2\pi]$. It is a topological invariant of the system; we refer to \cite{Wa} for a comprehensive introduction. In this section we prove the bulk-edge correspondence: the edge index of $\Pp_\delta$ can be computed from the bulk index of the operators $\Pp_{\delta,\pm}$.

\subsection{Description of the problem} Fix $\delta \in (0,\delta_\sharp)$ such that \eqref{eq:2b} holds. Define
\begin{equation}\label{eq:1x}
\delta_* \de \sup\Big\{ \delta > 0 : \ \forall \zeta \in [0,2\pi], \ \sup_{\tau \in [0,2\pi]} \lambda_{\delta,n}(\zeta k + \tau k') < \inf_{\tau \in [0,2\pi]} \lambda_{\delta,n}(\zeta k + \tau k') \Big\}.
\end{equation}
In \cite[Corollary 4]{Dr0} we showed that the spectral flow of $\Pp_{\delta}-E_\star$ is $-2\cdot \sgn(\te_\star)$ for every $\delta \in (0,\delta_*)$. Therefore, if $\delta_\flat \in (0,\delta_*)$ then Theorem \ref{thm:2} holds.

However, generally $\delta_* < \delta_\sharp$. This happens for instance when the no-fold condition of Fefferman--Lee-Thorp--Weinstein \cite[\S1.3]{FLTW3} fails. If $\zeta_\star = \lr{\xi_\star,v}$, the failure of the no-fold condition is equivalent to 
\begin{equation}
\exists \tau \in [0,2\pi], \ \  \zeta_\star k + \tau k' \notin \xi_\star + 2\pi \Lambda_*, \ \  \lambda_{0,n}(\zeta_\star k + \tau k') = E_\star \ \text{ or } \ \lambda_{0,n+1}(\zeta_\star k + \tau k') = E_\star.
\end{equation}
Since $\xi_\star \in \zeta_\star k + \R k'$,  in this situation the set involved in \eqref{eq:1x} is empty. Hence $\delta_* = -\infty$. As stated, Theorem \ref{thm:2} is more general than \cite[Corollary 4]{Dr0}. We will nonetheless derive Theorem \ref{thm:2} following the approach of \cite{Dr0}.

\begin{center}
\begin{figure}
\caption{The red curves represent sections of the $n$-th and $n+1$-th dispersion surfaces of $\Pp_{\delta,+}$ along $\zeta_\star k + \R k'$. When $\delta = \delta_\flat$, the operator $P_{\delta_\flat,+}[\zeta_\star]$ has a spectral gap at energy $E_\star + E_\flat(\zeta_\star)$ and the spectral flow of $\Pp_\delta-E_\star - E_\flat$ is well defined. As $\delta$ decreases to $0$, although the $n$-th and $n+1$-th dispersion surfaces of $\Pp_{\delta,+}$ remain separated, $\Pp_{\delta,+}[\zeta_\star]$ does not have an associated gap. However, $\PP_\delta[\zeta_\star] = \Pp_{\delta,+}[\zeta_\star] - \Tt_\delta[\zeta_\star]$ does.
}\label{fig:1}
{\includegraphics{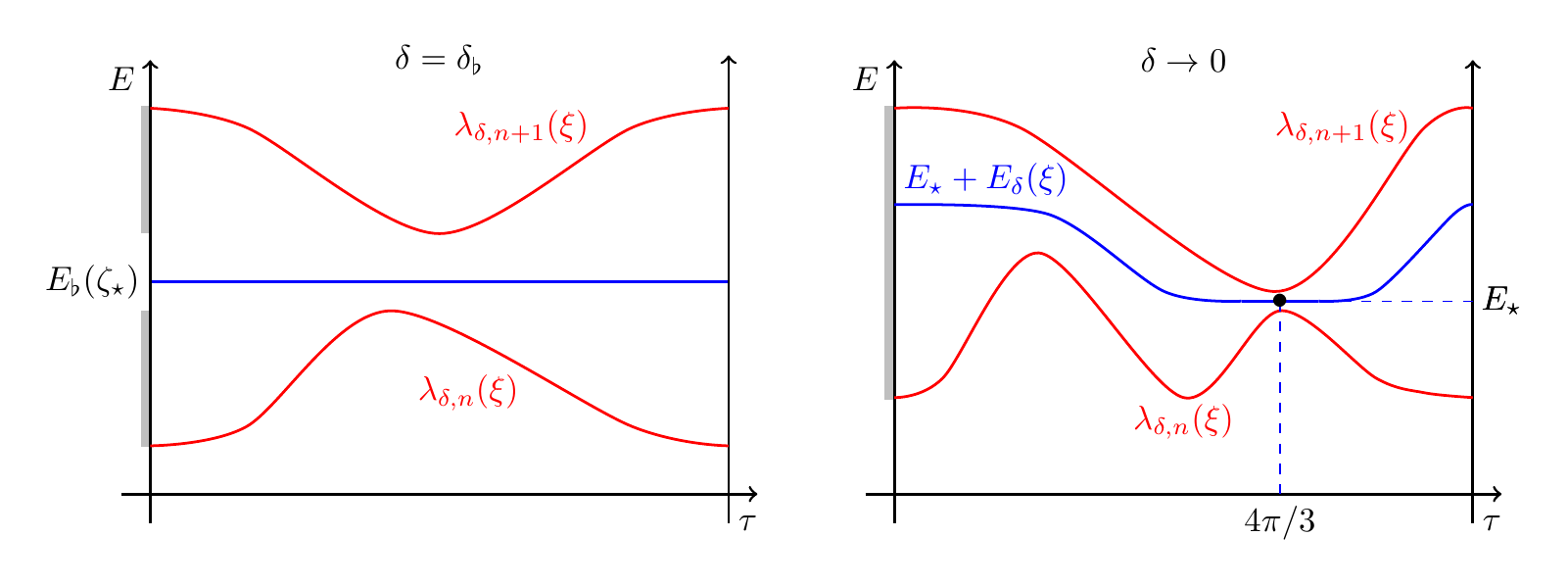}}
\end{figure}
\end{center}

\vspace*{-1cm}

As a preparation, we fix $E_\delta(\xi)$ depending smoothly on $(\delta,\xi) \in [0,\delta_\sharp) \times \R^2$, $2\pi\Lambda^*$-periodic in $\xi$, such that:
\begin{equations}\label{eq:2f}
\xi \notin \big\{\xi_\star^A,\xi_\star^B\big\} +2\pi \Lambda^*, \ \  \delta \in (0,\delta_\sharp) \ \ \Rightarrow \ \ 
\lambda_{\delta,n}(\xi) < E_\star + E_\delta(\xi) < \lambda_{\delta,n+1}(\xi);
\\
 E_\star + E_{\delta_\flat}(\xi) = E_\flat(\lr{\xi,v}); \ \ \ \text{ and } 
E_\delta(\xi) = 0 \text{ for $(\delta,\xi)$ near $\big(0,\xi_\star^A\big)$ and $\big(0,\xi_\star^B\big)$}.
\end{equations}
The first  condition is possible because of \eqref{eq:1a}; the second one is possible because of \eqref{eq:2b}. The third one is possible because $E_\star$ is not an eigenvalue of $\Pp_{\delta,+}(\xi)$ for $\xi$ near $\xi_\star^A$ and $\xi_\star^B$ and $\delta$ near $0$ -- see \eqref{eq:8c}.

Define $\Tt_\delta[\zeta]$ the operator formally equal to $E_\delta(D_x)$, but acting on $L^2[\zeta]$:
\begin{equation}\label{eq:1y}
\Tt_\delta[\zeta] \de \dfrac{1}{2\pi} \int_{[0,2\pi]}^\oplus E_\delta(\zeta k + \tau k') \cdot \Id_{L^2_{\zeta k+ \tau k'}} d\tau.
\end{equation}
Let $\PP_\delta[\zeta] = \Pp_\delta[\zeta] - \Tt_\delta[\zeta]$. Because of \eqref{eq:2f}, for $\delta \in (0,\delta_\sharp)$, $E_\star$ does not belong to the spectrum of the operators $\Pp_{\delta,\pm}(\xi) - E_\delta(\xi)$. Hence $\PP_\delta[\zeta]$ has an essential spectral gap at energy $E_\star$. We have the spectral flow equalities:
\begin{equation}\label{eq:2s}
\NN = \Sf\big(\Pp_{\delta_\flat} - E_\star- E_\flat \big) = \Sf\big(\Pp_{\delta_\flat} - E_\star - \Tt_{\delta_\flat}\big) = \Sf\big(\PP_{\delta_\flat} - E_\star\big)= \Sf\big(\PP_{\delta} - E_\star\big).
\end{equation}
The first equality is simply the definition of $\NN$. The second one comes from $\Tt_{\delta_\flat}[\zeta] = E_\flat(\zeta) \cdot \Id_{L^2_\zeta}$. Indeed, because of \eqref{eq:1y},
\begin{equation}
\Tt_\delta[\zeta] = \dfrac{1}{2\pi} \int_{[0,2\pi]}^\oplus E_\delta(\zeta k + \tau k') \cdot \Id_{L^2_{\zeta k+ \tau k'}} d\tau = \dfrac{E_\flat(\zeta)}{2\pi} \int_{[0,2\pi]}^\oplus   \Id_{L^2_{\zeta k+ \tau k'}} d\tau = E_\flat(\zeta) \cdot \Id_{L^2_\zeta}.
\end{equation}
The third equality in \eqref{eq:2s} is the definition of $\PP_{\delta_\flat}[\zeta]$; the last one holds because for $\delta \in (0,\delta_\sharp)$, $\PP_{\delta}[\zeta]-E_\star$ has a gap containing $0$, hence its spectral flow does not depend on $\delta$. Because of \eqref{eq:2s}, we can obtain $\NN$ by taking the limit of $\Sf\big(\PP_{\delta} - E_\star\big)$ as $\delta \rightarrow 0$.

We now follow the approach of \cite{Dr0}: we derive a resolvent estimate for $\PP_\delta[\zeta]$ as $\delta \rightarrow 0$.  The first step is an estimate on the bulk resolvent   $(\PP_{\delta,\pm}[\zeta]-E_\star)^{-1}$ where $\PP_{\delta,\pm}[\zeta] = \Pp_{\delta,\pm}[\zeta] + \Tt_\delta[\zeta]$ and $(\lambda,\zeta)$ is near $(E_\star,\zeta_\star)$, as in \cite[\S5]{Dr0}. Introduce:
\begin{equations}\label{eq:5r}
\RR : L^2\big(\R^2/\Z v, \C^2\big) \rightarrow L^2\big(\R,\C^2\big), \ \ \ \ \big(\RR f\big)(t) \de \int_0^1 f(sv + tv') ds; 
\\
\RR^* : L^2\big(\R, \C^2\big) \rightarrow L^2\big(\R^2/\Z v,\C^2\big), \ \ \ \ \big(\RR^* g\big)(x) \de g\big(\hspace*{-.8mm}\blr{k',x}\hspace*{-.8mm}\big);
\\
\UU_\delta : L^2\big(\R,\C^2\big) \rightarrow L^2\big(\R,\C^2\big), \ \ \ \ \big(\UU_\delta f\big)(t) \de  f(\delta t).
\end{equations}
Let $\Di(\mu)$ be the operator 
\begin{equation}\label{eq:1e}
\Di(\mu) \de  \matrice{\te_\star & \nu_\star k'   \\ \ove{\nu_\star k' } & -\te_\star}D_t + \mu \matrice{0 & \nu_\star \ell  \\ \ove{\nu_\star \ell} & 0 } + \matrice{\te_\star & 0 \\ 0 & -\te_\star} \kappa, \ \ \ \ \ell \de k-\dfrac{\lr{k,k'}}{|k|^2} k'.
\end{equation}
Above, $\nu_\star \ell$ is the complex number defined according to \eqref{eq:1i}.
We let $\Di_\pm(\mu) : H^1(\R,\C^2)$ $\rightarrow L^2(\R,\C^2)$ be the formal limits of $\Di(\mu)$ as $t \rightarrow \pm \infty$ -- i.e. replacing $\kappa$ in \eqref{eq:1e} by $\pm 1$. 
 
 \begin{theorem}\label{thm:3} Assume that \eqref{eq:2a} holds and that $\te_\star \neq 0$, $\zeta_\star = \lr{\xi_\star,v} \notin \pi \Z$. Fix $\tmu > 0$ and $\epsilon > 0$. There exists $\delta_0 > 0$ such that if
 \begin{equations}
 \delta \in (0,\delta_0), \ \ \mu \in (-\tmu,\tmu), \ \ z \in \Dd\left( 0, \sqrt{ \te_F^2 + \mu^2 \cdot \nu_F^2 | \ell|^2 } -\epsilon \right), \\
 \zeta = \zeta_\star + \delta \mu, \ \  \ \ \lambda = E_\star + \delta z
 \end{equations}
 then the operators $\PP_{\delta,\pm}[\zeta]- \lambda : H^2[\zeta] \rightarrow L^2[\zeta]$ are invertible. Furthermore, 
 \begin{equations}
 \left.\big(\PP_{\delta,\pm}[\zeta]- \lambda \big)^{-1} = \SSS_{\pm\delta}(\mu,z) + \OO_{L^2[\zeta]}\left(\delta^{-1/3}\right) \right., \\ (k' \cdot D_x) \big(\PP_{\delta,\pm}[\zeta]- \lambda \big)^{-1} = \SSS_{\pm\delta}^D(\mu,z) + \OO_{L^2[\zeta]}\left(\delta^{-1/3}\right),
\\
\text{where:} \ \ \ \  \SSS_{\pm\delta}(\mu,z) \de \dfrac{1}{\delta} \cdot \matrice{\phi_1 \\ \phi_2}^\top e^{i  \mu \delta \blr{\ell,x} } \RR^* \cdot  \UU_\delta \big( \Di_\pm(\mu) - z \big)^{-1} \UU_\delta^{-1} \cdot \RR e^{-i  \mu \delta \blr{\ell,x} } \ove{\matrice{\phi_1 \\ \phi_2}},
 \\
  \SSS_{\pm\delta}^D(\mu,z) \de \dfrac{1}{\delta} \cdot \matrice{(k' \cdot D_x)\phi_1 \\ (k' \cdot D_x)\phi_2}^\top e^{i  \mu \delta \blr{\ell,x} } \RR^* \cdot  \UU_\delta \big( \Di_\pm(\mu) - z \big)^{-1} \UU_\delta^{-1} \cdot \RR e^{-i  \mu \delta \blr{\ell,x} } \ove{\matrice{\phi_1 \\ \phi_2}}.
 \end{equations}
 \end{theorem}

\begin{proof} We explain why the proof of Theorem \ref{thm:3} is the same as \cite[Theorem 3]{Dr0}, without giving full details. There we processed with three main steps:
\begin{itemize}
\item We proved resolvent estimates on $L^2_\xi$ for $\xi \in \zeta k + \R k'$, away from $\xi_\star$;
\item We proved resolvent estimates on $L^2_\xi$ for $\xi \in \zeta k + \R k'$ near $\xi_\star$;
\item We integrated these estimates over the segment $\zeta k + [0,2\pi] \cdot k'$.
\end{itemize}

To reproduce the first step, we must check that $\PP_{\delta,\pm}(\xi)$ has a spectral gap near $E_\star$, when $\xi$ is away from $\xi_\star$ and $\delta$ is small. The eigenvalues of $\PP_{\delta,+}(\xi)$ are $\lambda_{\delta,j}(\xi) - E_\delta(\xi)$. Because of \eqref{eq:2f}, $\lambda_{\delta,n}(\xi) - E_\delta(\xi)$ gets closed to $E_\star$ only if $\xi$ approaches $\xi_\star^A$ or $\xi_\star^B$ modulo $2\pi \Lambda^*$. We must guarantee that $\xi_\star^A$ and $\xi_\star^B$ do not both belong to $\zeta_\star k + \R k' + 2\pi \Lambda^*$. This is equivalent to $\zeta_\star \notin \pi \Z$ --  which is assumed in Theorem \ref{thm:3}. Hence the first step in the proof of \cite[Theorem 3]{Dr0} goes through with only minor modifications: an analog of \cite[Lemma 4.1]{Dr0} holds. 

Since $E_\delta(\xi)$ vanishes near $\xi_\star$, adding the operator $E_\delta(\xi) \cdot \Id_{L^2_\xi}$ does not modify $\Pp_{\delta,+}(\xi)$ for $\xi$ near $\xi_\star$. Thus the second step in the proof of \cite[Theorem 3]{Dr0} is unchanged.

Because the first and second step lead to the same results as in \cite[\S 4]{Dr0}, the third step (the integration process) is identical. This completes the proof of Theorem \ref{thm:3}.
\end{proof}

As in \cite[\S6]{Dr0}, we use the bulk resolvent estimates of Theorem \ref{thm:3} to derive resolvent estimate for the edge operator $\PP_\delta[\zeta]$. We introduce a parametrix:
\begin{equation}
\QQ_\delta(\zeta,\lambda) \de  \sum_\pm \chi_{\pm,\delta} \cdot (\PP_{\delta,\pm}[\zeta]-\lambda)^{-1}, \ \ \ \ \chi_{\pm,\delta} \de \dfrac{1 \pm \kappa_\delta}{2}.
\end{equation}
A calculation shows:
\begin{equations}\label{eq:3e}
\big( \PP_\delta[\zeta]-\lambda \big) \cdot \QQ_\delta(\zeta,\lambda) - \Id 
= \sum_{\pm} \big( \PP_\delta[\zeta]-\lambda \big) \cdot \chi_{\pm,\delta} \cdot \big(\PP_{\delta,\pm}[\zeta]-\lambda\big)^{-1} - \Id
\\
= \sum_{\pm} \big( \PP_{\delta,\pm}[\zeta]-\lambda + \kappa_\delta \cdot \delta\Ww \mp \delta \Ww \big) \cdot \chi_{\pm,\delta} \cdot \big(\PP_{\delta,\pm}[\zeta]-\lambda\big)^{-1}
\\
= \sum_{\pm} \mp \dfrac{1 - \kappa_\delta^2}{2} \cdot \delta\Ww \cdot \big(\PP_{\delta,\pm}[\zeta]-\lambda\big)^{-1} + \sum_{\pm} \left[D_x^2 + \Tt_\delta[\zeta],\chi_{\pm,\delta}\right] \cdot \big(\PP_{\delta,\pm}[\zeta]-\lambda\big)^{-1}
\\
= \sum_{\pm} \left(\left[D_x^2,\chi_{\pm,\delta}\right] \mp \dfrac{1 - \kappa_\delta^2}{2}  \cdot \delta\Ww\right)  \cdot \big(\PP_{\delta,\pm}[\zeta]-\lambda\big)^{-1} + \sum_{\pm} \big[\Tt_\delta[\zeta],\chi_{\pm,\delta}\big] \cdot \big(\PP_{\delta,\pm}[\zeta]-\lambda\big)^{-1}.
\end{equations}
The next lemma proves that the terms $\big[\Tt_\delta[\zeta],\chi_{\pm,\delta}\big] \cdot (\PP_{\delta,\pm}[\zeta]-\lambda)^{-1}$ are negligible. 

\begin{lem}\label{lem:1f} Assume that the conditions of Theorem \ref{thm:3} are satisfied. Then 
\begin{equation}
\big[\Tt_\delta[\zeta],\chi_{\pm,\delta}\big] \cdot \big(\PP_{\delta,\pm}[\zeta]-\lambda\big)^{-1} = \OO_{L^2[\zeta]}(\delta^{2/3}).
\end{equation}
\end{lem}

The basic idea is that because of Theorem \ref{thm:3}, $\big(\PP_{\delta,\pm}[\zeta]-\lambda\big)^{-1}$ localizes to frequencies near $\xi_\star$ modulo lower order terms; while $\Tt_\delta[\zeta]$ essentially localizes to frequencies away from $\xi_\star$.
Semiclassical analysis provides the natural tool to prove Lemma \ref{lem:1f}. We use the notations of \cite[\S4]{Z}. We say that a smooth function  $(t,\tau) \in \R^2 \mapsto a(t,\tau) \in \C$ (possibly depending on $\delta \in (0,\delta_\flat]$) belongs to the symbol class $S$  when:
\begin{equation}
\forall \az, \beta \in \N, \ \ \ \ \sup \left\{ \left|\p_x^\az \p_\xi^\beta a(t,\tau)\right| : (t,\tau,\delta) \in \R^2 \times (0,\delta_\flat] \right\} < \infty.
\end{equation}
See \cite[\S4.4]{Z}.
For $a \in S$, we denote by $a^W$ the Weyl quantization of $a$ with semiclassical parameter $\delta$ -- see \cite[(4.1.1)]{Z}. This is a bounded operator on $L^2$ -- see \cite[Theorem 4.23]{Z}. Moreover, if $b \in S$, then
\begin{equation}\label{eq:1o}
a^W b^W = (ab)^W + \dfrac{\delta}{2i}\{a,b\}^W + \OO_{L^2}\big(\delta^2\big),
\end{equation}
where $\{a,b\}$ is the Poisson bracket of $a$ and $b$.
The formula \eqref{eq:1o} follows from \cite[Theorem 4.18 and (4.4.15)]{Z} which writes $a^W b^W$ as a semiclassical operator with symbol
\begin{equation}\label{eq:1n}
ab + \dfrac{\delta}{2i}\{a,b\} + O_S\big(\delta^2\big);
\end{equation}
and \cite[Theorem 4.23]{Z}: the quantization of a symbol $O_S\big(\delta^2\big)$ is $\OO_{L^2[\zeta]}\big(\delta^2\big)$.

\begin{proof}[Proof of Lemma \ref{lem:1f}] 1.  Let $\chi, \Psi \in C^\infty(\R,\C)$ bounded together with their derivatives, with uniform bounds as $\delta$ goes to zero. We observe that 
\begin{equation}\label{eq:1r}
\big[\Psi(D_t), (\UU_\delta\chi) \big] = \UU_\delta \big[\Psi(\delta D_t),\chi\big] \UU_\delta^{-1}.
\end{equation}
Note that $\Psi(\delta D_t)$ is a semiclassical pseudodifferential operator with symbol $(t,\tau) \mapsto \Psi(\tau)$; and $\chi$ is also a semiclassical pseudodifferential operator with symbol $(t,\tau) \mapsto \chi(t)$ because of \cite[(4.1.6)]{Z}. We deduce from \eqref{eq:1o} that
\begin{equation}\label{eq:2p}
\big[\Psi(\delta D_t),\chi\big] =
\dfrac{\delta}{i} \big\{ \Psi, \chi \big\}^W + \OO_{L^2}\big(\delta^2\big).
\end{equation}
In particular, the operator \eqref{eq:1r} is $\OO_{L^2}(\delta)$.

2. Assume that in addition, $\Psi$ vanishes in a $\delta$-independent neighborhood of $0$. Then we can write \eqref{eq:2p} as
\begin{equation}
\big[\Psi(\delta D_t),\chi\big] =
\delta \cdot  \left(\dfrac{\big\{ \Psi(\tau), \chi(t) \big\}}{i\tau} \cdot \tau \right)^W + \OO_{L^2}\big(\delta^2\big).
\end{equation}
We use \eqref{eq:1o} to deduce that 
\begin{equation}
\big[\Psi(\delta D_t),\chi\big] =
\delta \cdot \left(\dfrac{\big\{ \Psi(\tau), \chi(t) \big\}}{i\tau} \right)^W \cdot \delta D_t + \OO_{L^2}\big(\delta^2\big).
\end{equation}
Thanks to \eqref{eq:1r}, if $\Psi$ vanishes in a $\delta$-independent neighborhood of $0$ then
\begin{equation}\label{eq:1t}
\big[\Psi(D_t),(\UU_\delta \chi)\big] \UU_\delta =
\delta \cdot \UU_\delta \left(\dfrac{\big\{ \Psi(\tau), \chi(t) \big\}}{i\tau} \right)^W \cdot \delta D_t + \OO_{L^2}\big(\delta^2\big) = \OO_{H^1 \rightarrow L^2}\big(\delta^{3/2}\big).
\end{equation}

3. For $f \in L^2[\zeta]$, set $F(s,t) = f(sv + tv')$. If $\GG$ is a bounded operator on $L^2(\R)$, we define 
\begin{equation}
\Gg f(sv+tv') = \big(\GG F(s,\cdot)\big)(t).
\end{equation}
Observe that $\|\Gg\|_{L^2[\zeta]} \leq \|\GG\|_{L^2}$:
\begin{equations}
|\Gg f|_{L^2[\zeta]}^2 = \int_0^1 \int_\R \big| \Gg f(sv+tv') \big|^2 ds dt = \int_0^1 \int_\R \big| \GG F(s,\cdot) \big)(t) \big|^2 ds dt
\\
= \int_0^1 \big| \GG F(s,\cdot) \big|^2_{L^2} ds \leq \|\GG\|_{L^2}^2 \cdot \int_0^1 \big| F(s,\cdot) \big|^2_{L^2} ds = \|\GG\|_{L^2}^2 \cdot |f|^2_{L^2[\zeta]}.
\end{equations}

We now observe that if $f \in L^2[\zeta]$, then
\begin{equation}\label{eq:8d}
\big(\big[\Tt_\delta[\zeta],\chi_{\pm,\delta}\big]f\big) (sv+tv') =\Big( \UU_\delta^{-1} \cdot \big[E_\delta\big(\zeta k + \delta k'D_t\big),\chi_{\pm}\big] \cdot \UU_\delta F(s,\cdot)\Big)(t).
\end{equation}
To prove \eqref{eq:8d}, we fix $f \in L^2[\zeta]$ which we expand in Fourier series w.r.t. $\Z v$:
\begin{equation}
f(x) = \sum_{m \in 2\pi\Z} e^{i(\zeta + m)\lr{k,x}} f_m(\lr{k',x}).
\end{equation}  
Since $e^{i(\zeta + m)\lr{k,x}} \in L^2_{(\zeta + m)k}$, and $E_\delta(\xi)$ depends periodically on $\xi$, $E_\delta(D_x) e^{i(\zeta + m)\lr{k,x}} = e^{i(\zeta + m)\lr{k,x}}  E_\delta(\zeta k + D_x)$. It follows that
\begin{equation}
(E_\delta(D_x) f\big)(sv+tv') = \sum_{m \in 2\pi \Z} e^{i(\zeta + m)s} \cdot \big(E_\delta(\zeta k + k' D_t) f_m\big)(t) = \big(E_\delta(\zeta k + k' D_t) F(s,\cdot)\big)(t).
\end{equation}
It suffices to recall that $\Tt_\delta$ is formally equal to $E_\delta(D_x)$ (though acting on $L^2[\zeta]$) to conclude the proof of \eqref{eq:8d}.

We now apply Step 1 to $\Psi(\tau) = E_\delta(\zeta k + \tau k')$ and $\chi = \chi_\pm = \frac{1 \pm \kappa}{2}$. We deduce that
\begin{equation}
\big\| \big[\Tt_\delta[\zeta],\chi_{\pm,\delta}\big] \big\|_{L^2[\zeta]} \leq \big\| \UU_\delta^{-1} \cdot \big[E_\delta\big(\zeta k + \delta k'D_t\big),\chi_{\pm}\big] \cdot \UU_\delta \big\|_{L^2} = O(\delta).
\end{equation}
Using Theorem \ref{thm:3}, we see that
\begin{equations}
\big[\Tt_\delta[\zeta],\chi_{\pm,\delta}\big] \cdot (\PP_{\delta,\pm}[\zeta]-\lambda)^{-1}
= \big[\Tt_\delta[\zeta],\chi_{\pm,\delta}\big] \cdot \SSS_\delta(\mu,z) + \OO_{L^2[\zeta]}\big(\delta^{2/3}\big).
\end{equations}

4. To conclude the proof, we show that $\big[\Tt_\delta[\zeta],\chi_{\pm,\delta}\big] \cdot \SSS_\delta(\mu,z) = \OO_{L^2[\zeta]}(\delta)$. Let $\Phi(x) = e^{i\mu\delta \lr{\ell,x}} [\phi_1(x), \phi_2(x)]^\top$. 
We write  $\big[\Tt_\delta[\zeta],\chi_{\pm,\delta}\big] \cdot \SSS_\delta(\mu,z) = \delta^{-1} T_{\pm,1}  \cdot T_{\pm,2}$, where 
\begin{equations}
T_{\pm,1} \de \big[\Tt_\delta[\zeta],\chi_{\pm,\delta}\big] \cdot \Phi  \RR^* \cdot  \UU_\delta \ : \ H^1 \rightarrow L^2[\zeta],
\\
T_{\pm,2} \de  \big( \Di_\pm(\mu) - z \big)^{-1} \UU_\delta^{-1} \cdot \RR e^{-i  \mu \delta \blr{\ell,x} } \ove{\matrice{\phi_1 \\ \phi_2}} \ : \ L^2[\zeta] \rightarrow H^1.
\end{equations}
We observe that $T_{\pm,2} = \OO_{L^2[\zeta] \rightarrow H^1}\big(\delta^{1/2}\big)$.

We recall that $\Tt_\delta[\zeta]$ is the operator formally equal to $E_\delta(D_x)$ but acting on $L^2[\zeta]$. Thanks to this identification, we have
\begin{equation}
\big[\Tt_\delta[\zeta],\chi_{\pm,\delta}\big] \cdot \Phi = \big[E_\delta(D_x),\chi_{\pm,\delta}\big]  \cdot \Phi
= \Phi  \cdot \big[E_\delta(D_x - \xi_\star - \mu \delta k),\chi_{\pm,\delta}\big].
\end{equation}
Above, we used that $E_\delta$ is periodic and that $\Phi \in L^2_{\xi_\star-\mu \delta k}$. We deduce that
\begin{equations}
T_{\pm,1} =  \Phi  \cdot \big[E_\delta(D_x - \xi_\star - \mu \delta k),\chi_{\pm,\delta}\big] \RR^* \UU_\delta
=  \Phi \RR^*  \cdot \big[E_\delta(k'D_t - \xi_\star - \mu \delta k),(\UU_\delta \chi_{\pm})\big]  \UU_\delta.
\end{equations}
We now apply \eqref{eq:1t} with $\chi = \chi_\pm$ and $\Psi(\tau) = E_\delta(\tau k' - \xi_\star - \mu \delta k)$; we observe that $\Psi$ vanishes in a $\delta$-independent neighborhood of $0$ because $E_\delta(\xi)$ vanishes when $\xi$ is near $\xi_\star$. We deduce that  $T_{\pm,1} = \OO_{H^1 \rightarrow L^2[\zeta]}\big(\delta^{3/2}\big)$. Since $\big[\Tt_\delta[\zeta],\chi_{\pm,\delta}\big] \cdot \SSS_\delta(\mu,z) = \delta^{-1} T_{\pm,1}  \cdot T_{\pm,2}$, we deduce that $\big[\Tt_\delta[\zeta],\chi_{\pm,\delta}\big] \cdot \SSS_\delta(\mu,z) = \OO_{L^2[\zeta]}\big(\delta^2\big)$. The proof of the lemma is complete.\end{proof}

We conclude from Lemma \ref{lem:1f} and the discussion preceding it that
\begin{equations}
\big( \PP_\delta[\zeta]-\lambda \big) \cdot \QQ_\delta(\zeta,\lambda) =  \Id + \KK_\delta(\zeta,\lambda)+ \OO_{L^2[\zeta]}\big(\delta^{2/3}\big),
\\ 
\KK_\delta(\zeta,\lambda) \de \sum_{\pm} \left(\big[D_x^2,\chi_{\pm,\delta}\big] \mp \dfrac{1 - \kappa_\delta^2}{2}  \cdot \delta\Ww \right) \cdot \big(\PP_{\delta,\pm}[\zeta]-\lambda\big)^{-1}.
\end{equations}
The operator $\KK_\delta(\zeta,\lambda)$ and $\QQ_\delta(\zeta,\lambda)$
satisfy the same expansions as $\KKK_\delta(\zeta,\lambda)$ and $\QQQ_\delta(\zeta,\lambda)$ in \cite[\S6.1]{Dr0}, because Theorem \ref{thm:3} provides the same resolvent estimates as \cite[Theorem 3]{Dr0}. Therefore, the proof of \cite[Theorem 2]{Dr0} applies without further changes. It yields:

\begin{theorem} Assume that \eqref{eq:2a} holds and that $\te_\star \neq 0$, $\zeta_\star \notin \pi \Z$; 
fix $\tmu > 0$ and $\epsilon > 0$. Let $\Sigma(\mu)$ denote the $L^2$-spectrum of $\Di(\mu)$. There exists $\delta_0 > 0$ such that if
\begin{equations}
\mu \in (-\tmu,\tmu), \ \ \delta \in (0,\delta_0), \  \ z \in \Dd\Big(0,\sqrt{\te_F^2 + \mu^2\cdot \nu_F^2 | \ell|^2} -\epsilon \Big), \ \ 
 \dist\big( \Sigma(\mu), z \big) \geq \epsilon, 
\\
 \zeta = \zeta_\star +  \delta\mu, \ \ \lambda = E_\star + \delta z 
\end{equations}
then $\PP_\delta[\zeta] - \lambda$ is invertible and its resolvent $\big( \PP_\delta[\zeta] - \lambda \big)^{-1}$ equals
\begin{equation}
 \dfrac{1}{\delta} \cdot  \matrice{ \phi_1 \\ \phi_2}^\top  e^{-i  \mu \delta \blr{\ell,x} } \cdot \RR^*   \UU_\delta \cdot \big( \Di(\mu)-z \big)^{-1}  \cdot\UU_\delta^{-1}  \RR \cdot e^{i  \mu \delta \blr{\ell,x} } \ove{\matrice{\phi_1 \\ \phi_2}} + \OO_{L^2[\zeta]}\left(\delta^{-1/3}\right).
\end{equation}
\end{theorem}

The family $\mu \mapsto \Di(\mu)$ has spectral flow equal to $-\sgn(\te_\star)$ as $\mu$ runs through $\R$ -- see \cite[\S3.2]{Dr0}. Since there are two Dirac points, we recover a spectral flow of $\PP_\delta-E_\star$ equal to $-2 \cdot \sgn(\te_\star)$ for small $\delta$ -- see the proof of \cite[Corollary 4]{Dr0}. The identity \eqref{eq:2s} completes the proof of Theorem \ref{thm:2}.

\end{document}